\documentclass{article}
\RequirePackage{etex}
\usepackage{amsmath}
\usepackage{amsthm}
\usepackage{amsfonts}
\usepackage{tcolorbox}
\usepackage{graphicx}
\usepackage{pgfplots}
\usepackage{tikz}
\usepackage{authblk}
\usepackage{tikz-network}
\newtheorem{introtheorem}{Theorem}

\title{Discrete trace formulas and holomorphic functional calculus for the adjacency matrix of regular graphs}
\author[1]{Yulin Gong\thanks{
Current address: School of Mathematics, University of Bristol, BS8 1UG, UK. \\
\indent \quad Email: dz25829@bristol.ac.uk}}\author[2]{Wenbo Li}\author[3]{Shiping Liu}
\affil[1]{Department of Mathematical Sciences, Tsinghua University, Beijing 100084, China}
\affil[2,3]{School of Mathematical Sciences, University of Science and Technology of China, Hefei 230026, China}
\affil[1]{gongyl22@mails.tsinghua.edu.cn}
\affil[2]{patlee@mail.ustc.edu.cn}
\affil[3]{spliu@ustc.edu.cn}
\date{}
\pgfplotsset{compat=1.18}
\begin{document}
\maketitle

\newtheorem{theorem}{Theorem}[section]
\newtheorem{proposition}{Proposition}[section]
\newtheorem{corollary}{Corollary}[section]
\newtheorem{lemma}{Lemma}[section]
\newtheorem{definition}{Definition}[section]
\newtheorem{remark}{Remark}[section]
\newtheorem{eg}{Example}[section]
\newtheorem{Notation}{Notation}[section]

\begin{abstract}
  We provide a unified method to study the adjacency matrices of regular graphs (including infinite ones) using holomorphic functional calculus. By applying this calculus on a specific ellipse that contains the spectrum, we derive an expansion of $h(A)$ using non-backtracking matrices. This framework allows us to systematically obtain discrete trace formulas that link spectral theory with graph combinatorics. To show how this method works, we give new proofs for several well-known problems, such as walk counting, the Ihara-Bass formula, and solutions to the heat and Schrödinger equations on graphs.
\end{abstract}

\tableofcontents

\section{Introduction}
\subsection{Background}
For any finite graph, the spectrum of its adjacency matrix $A$ is related to its closed walks by the trace method. However, a general walk can be extremely complicated due to \emph{backtracking}. We consider the ``simpler" walk called a \emph{non-backtracking walk}; see Definition \ref{defofnbws}. The non-backtracking walks play an important role in the study of the spectrum of the adjacency matrix of a regular graph, see, e.g., Lubotzky, Phillips, and Sarnak \cite{LPS88} for constructions of Ramanujan graphs, Friedman \cite{Fri91,Fri08}, Bordenave \cite{Bor20} for proofs of Alon's second eigenvalue conjecture.

In this article, we consider regular graphs with given vertex degree. As the number of vertices tends to infinity, such a regular graph becomes sparse, since the number of edges grows linearly. We notice that on a sparse graph, the number of non-backtracking walks is much smaller than the number of general walks. Especially on a tree, there is only one non-backtracking walk between given two vertices and orientation.

We study discrete trace formulas to relate the spectrum and non-backtracking walks on regular graphs. Our motivation is Selberg's pioneering work \cite{Sel56} well-known as \emph{Selberg's trace formula}, establishing a deep relation between the spectrum of Laplacian and the closed geodesics of hyperbolic surfaces, see \eqref{Seltrace} below. 
There has been great interest in the study of discrete trace formulas on regular graphs, see, e.g.,  Ahumada \cite{Ahu87},  Brooks \cite{Brooks91}, Venkov and Nikitin \cite{VN94}, Terras and Wallace \cite{TW03}. For more related works, see Horton, Newland and Terras \cite{HNT06}, Stark and Terras \cite{ST96, ST00, ST07}, Sunada \cite{Su06}, Terras \cite{Terras02}, and the book \cite{HFGO12} and references therein. Recall that these previous works on discrete trace formulas are based on a study of the geometry (geodesics, horocycles, isomorphism group) of regular trees in analogue to that of hyperbolic spaces, and treating 
a regular graph as a quotient of a regular tree. Trace formulas obtained in this way apply to functions constructed from rotation-invariant functions on regular trees.

However, for applications it is natural to acquire that the trace formula can be applied to any possible functions without knowing a priori how it is explicitly constructed from the rotation-invariant functions on trees. By applying the Cauchy's integral formula, we generalize the trace formulas as studied in previous works to a functional calculus formula for adjacency matrix (operator). The same method also leads to a useful trace formula for any function that is analytic on the interior of a particular ellipse containing the spectrum of the graph. Notice that the functional calculus formula has a much broader use in practice, as one could expect.

. 

\subsection{Main theorem}
Building upon a combinatorial observation of Friedman \cite[Lemma 3.3]{Fri91}, we give an elementary construction of discrete trace formulas for a $(q+1)$-regular graph $G=(V,E)$. The non-backtracking walks are intrinsically related with Chebyshev-type polynomials $X_{r,q}(x)$. More precisely, there holds
\begin{equation}\label{eq:introArXrq}
    A_{r}=q^{r/2}X_{r,q}(q^{-1/2}A).
\end{equation}
where $A_r$ is the non-backtracking matrix of length $r$ defined by $$(A_{r})_{ab}=\# \{\text{non-backtracking walks from} \ a \ \text{to} \ b \text{ of length } r\}.$$

We first apply the generating functions of Chebyshev-type polynomials $X_{r,q}$ to derive a formula for the Stieltjes transform of the normalized spectral measure $\mu_{G}$ in terms of the number of non-backtracking walks (see Propositions \ref{stieltjesq} and \ref{Iharababy}). Indeed, it is derived from a fact on decomposing the resolvent of $A$ in terms of the non-backtracking matrices (see Lemma \ref{green}). The resolvent and Stieltjes transform are quite useful in the study of spectral theory of random regular graphs, see for example Huang, McKenzie, and Yau \cite{HMY24}.

Next, we establish the following functional calculus formula.
Let us denote by $\mu_q$ the Kesten-McKay distribution, which is the normalized spectral measure of the $(q+1)$-regular tree. 
\begin{introtheorem}[Functional calculus formula]\label{introexpansion}
    Let $G=(V,E)$ be a possibly infinite $(q+1)$-regular graph with adjacency matrix (operator) $A$.
    Suppose $h$ is holomorphic on the domain $\Omega(\rho)$ for $\rho>q^{1/2}$ defined below
       $$\Omega(\rho):=\left\{w\in \mathbb{C}:\frac{(\mathrm{Re}\ w)^2}{(\rho+\rho^{-1})^2}+\frac{(\mathrm{Im}\ w)^2}{(\rho-\rho^{-1})^2}<1\right\}.$$
    Then the following formula holds
    \begin{equation}\label{eq:introh}
        h(q^{-1/2}A)=\int_{\mathbb{R}} h(x)d\mu_q(x)\mathrm{I}+\sum_{r=1}^{\infty}q^{-r/2}a_{r,q}(h) A_r,
    \end{equation}
    where 
    \begin{equation}\label{eq:introa}
        a_{r,q}(h)=(1+q^{-1})^{-1}\int_{-2}^{2}h(x)X_{r,q}(x)d\mu_q(x).
    \end{equation}
\end{introtheorem}
The above formula serves as a germ for our pre-trace and trace formulas. To elaborate the key ingredients of our method, we describe the main idea of the proof for Theorem \ref{introexpansion} here.
Indeed, any holomorphic function $h$ on $\Omega_\rho$ can be expanded via Chebyshev-type polynomials $X_{r,q}$ as
\begin{equation}\label{eq:hexpansionChebyshev}
    h=\sum_{r=0}^\infty a_{r,q}(h)X_{r,q},
\end{equation}
where the coefficients $a_{r,q}(h),\,r\geq 0$ are derived by applying Cauchy's integral formula to $h$ and employing the generating function for $X_{r,q},\,r\geq 0$ (see Theorem \ref{master} and its proof in Section \ref{section:proof}). Another important fact is that $X_{r,q},\,r\geq 0$ form a complete orthogonal basis for $L^2(\mathbb{R};\mu_q)$ (see Lemma \ref{lemma:orthogonal}). Apply the expansion \eqref{eq:hexpansionChebyshev} to the matrix $q^{-1/2}A$ yields 
\begin{equation*}
    h(q^{-1/2}A)=\sum_{r=0}^\infty a_{r,q}(h)X_{r,q}(q^{-1/2}A).
\end{equation*}
Then the functional caculus formula \eqref{eq:introh} follows directly by the fact $X_{0,q}\equiv 1$ and the relation \eqref{eq:introArXrq}. The formula \eqref{eq:introa} for $a_{r,q}(h)$ is derived straightforwardly from the orthogonality of polynomials $X_{r,q}$ with respect to the Kesten-McKay distribution $\mu_q$.

Notice that the spectral radius of $q^{-1/2}A$ is no greater than $q^{1/2}+q^{-1/2}$.
The fact that $h$ is holomorphic on $\Omega(\rho)$ ensures the uniform convergence of $\sum_{r=0}^{\infty}a_{r,q}X_{r,q}$ on the closed ellipse $\overline{\Omega(q^{1/2})}$ (see Theorem \ref{master}). Thus, we obtain the following pre-trace formula immediately. We denote by $\mu_G^v$ the normalized spectral measure of the graph $G$ at a vertex $v$ (see Definition \ref{def:spectralmeasureatvertex}). 
\begin{introtheorem}[Discrete pre-trace formula]\label{intropretrace}
    Let $G=(V,E)$ be a possibly infinite $(q+1)$-regular graph. Suppose 
    $h$ is holomorphic on $\Omega(\rho)$ for $\rho>q^{1/2}$. For any $v \in V$, the following pre-trace formula holds
    \begin{equation*}
\int_{\mathbb{R}} h(x) d\mu_{G}^{v}=
\int h(x) d\mu_q(x)+\sum_{r=1}^{\infty}q^{-r/2}a_{r,q}(h)f_r(v;G),
\end{equation*}
where $$a_{r,q}(h)=(1+q^{-1})^{-1}\int_{-2}^{2}h(x)X_{r,q}(x)d\mu_q(x),$$
and $f_{r}(v;G)$ stands for the number of closed non-backtracking walks from (and to) the vertex $v$.
\end{introtheorem}
 Now we aim at calculating the trace of $h(A)$ of a finite $(q+1)$-regular graph by the number of closed non-backtracking walks. We treat closed non-backtracking walks as discrete analogues of geodesic loops. The discrete analogues of closed geodesics should be \emph{circuits}. A circuit is a closed non-backtracking walk whose final edge is not the inverse of the first edge. Hence the cyclic rotation of a circuit is again a circuit. The circuits are related to the Chebyshev-type polynomials $Y_{r}:=X_{r,1}$ (see \eqref{circuitnumber}). 
Then we express an admissible test function $h$ by $Y_{r}$ to obtain the following discrete trace formula. We denote by 
\[\lambda_{|V|}(A) \leq \cdots \leq \lambda_2(A) \leq \lambda_1(A)\]
the eigenvalues of the adjacency matrix $A$ of a finite graph $G=(V,E)$.
\begin{introtheorem}[Discrete trace formula]\label{introdiscretetracecircuit}
    Let $G=(V,E)$ be a finite $(q+1)$-regular graph with adjacency matrix $A$. Suppose $h$ is holomorphic on $\Omega(\rho)$ for $\rho>q^{1/2}$,
    the following discrete trace formula holds
    \begin{equation*}\label{discretetraceintro}
\sum_{k=1}^{|V|} h(q^{-1/2}\lambda_k(A))=
\int_{\mathbb{R}} h(x) d\mu_q(x)|V|+\sum_{r=1}^{\infty}q^{-r/2}a_{r,1}(h)c_r(G),
\end{equation*}
    where 
    \begin{equation*}
        a_{r,1}(h)=\frac{1}{2\pi i}\oint_{\partial B(0,1)} h(\xi+\xi^{-1})\xi^{r-1}d\xi
    \end{equation*}
is the $r$-th coefficient of the expansion $h=\sum_{r=0}^\infty a_{r,1}(h)Y_r$, and $c_r(G)$ is the number of circuits in $G$. 
\end{introtheorem}
Furthermore, we introduce the equivalent classes of prime circuits in $G$ to reformulate our trace formula in Theorem \ref{introdiscretetracecircuit} into a form which resembles perfectly the Selberg's trace formula on hyperbolic surfaces in appearance, see Theorem \ref{Discretetraceformulaprime version} and Remark \ref{rmk:comparison} for more detailed discussion and comparison. 

We discuss several applications of the above theorems. First, we derive a formula for the number of walks on a regular graph in terms of the number of non-backtracking walks, see Proposition \ref{numberofwalkonrg}. This leads to a new counting formula for the number of walks on a regular tree, see Corollary \ref{numberofwalktree}. Next, we study the heat and Schr\"odinger equations on regular graphs. We establish explicit formulas for the corresponding fundamental solutions in terms of non-backtracking matrices, see Proposition \ref{fdsolutionofheat}. We further apply the results on heat equations to lattices and obtain a counting formula for the number of walks on lattices, see Proposition \ref{heatkerneloflattice} and Corollary \ref{numberofwalklattice}. At last, we apply our trace formula to logarithm and exponential functions to derive Ihara-Bass theorem for Ihara zeta function,  see Theorem \ref{iharabasseq}, Neumann expansion for Fourier-Laplace transform of spectral measure, see Theorem \ref{thm:fourierlaplace}, and heat trace formula, see Corollary \ref{heattraceformula}.

\subsection{Comments on Ihara zeta function, heat and Schr\"odinger equations on graphs}
We briefly survey the history on the Ihara zeta function. In the work \cite{Sel56}, 
Selberg's trace formula is deeply related to Selberg zeta function.  In \cite{Ihara66}, Ihara provided a $p$-adic version of Selberg's zeta function. Serre \cite{Ser02} suggested that Ihara's definition can be interpreted in the language of graph theory. Sunada \cite{ST06} put this suggestion into practice. In \cite{BH92}, Bass proved the so-called \emph{Ihara-Bass theorem}. There is lots of work on reproofs and generalizations of Ihara-Bass theorem, see Kotani-Sunada \cite{KS00}, Deitmar \cite{Dei15}, Mitsuhashi-Morita-Sato \cite{KMMS19}, Rangarajan \cite{RB18}, Chinta, Jorgenson, and Karlsson \cite{CJK15}, Stark and Terras \cite{ST96, ST00}, see also the book \cite{HFGO12} and references therein.

An important related topic is Ramanujan graphs. A \emph{Ramanujan graph} is a $(q+1)$-regular graph whose nontrivial eigenvalues fall in the interval \[[-2q^{1/2},2q^{1/2}].\]   
By Ihara-Bass Theorem, a regular graph is a Ramanujan graph if and only if its Ihara zeta function satisfies an analogue of the Riemann hypothesis. Ramanujan graphs are best possible expanders by Alon-Boppana bounds \cite{Nilli91}. We refer to \cite{Mur03,HLW06} for the theory of expanders, and to \cite{LPS88,Mar88,Mor94,MMS13,MMS18} for constructions of Ramanujan graphs.

The heat equation on a graph models the continuous-time random walk.  Both the heat equation and the heat trace of a regular graph have been extensively studied and have several applications in estimating eigenvalues and combinatorial quantities of a graph, see, e.g. Grigor'yan \cite{Grigo09} for heat kernel estimate and Faber-Krahn inequality, Chung-Yau \cite{CY97, CY99} for heat trace formulas on lattice graphs, regular trees and estimates of the number of spanning trees, Mn\"ev \cite{Mnev07}, Horton, Newland and Terras \cite{HNT06}, Chinta, Jorgenson, and Karlsson \cite{CJK15}, Jorgenson, Karlsson, and Smajlovi{\'c} \cite{JKS24} for heat kernel and heat trace formulas on regular graphs. 

The Schr\"odinger equation on regular graphs and especially lattices has drawn a lot of attention throughout the years. In the literature of physics, discrete Schr\"odinger equation corresponds to the tight-binding model of solids. Various topics of mathematical physics such as the mathematically rigorous approach of the Anderson model can be reduced to the study of Schr\"odinger equation on graphs. We refer to \cite{stolz2011introduction} for an introduction to the mathematical theory of the Anderson model.

\section{Preliminaries}

In this section, We collect the preliminaries needed for this paper. 

\subsection{Non-backtracking walks and matrices}\label{NBM}

Let $G=(V,E)$ be an undirected graph. We allow the edge set $E$ to contain multi-edges and multi-loops.  Let $\vec{E}$ be the set of directed edges obtained from $E$, such that each element of $E$ corresponds to two distinct directed edges in $\vec{E}$. For every directed edge $e \in \vec{E}$,  denote the origin (resp., terminus) of $e$ as $o(e)$ (resp., $t(e)$)  and denote its inverse as $\overline{e}$. 

By a \textbf{walk} from $a$ to $b$, we mean a sequence of vertices and edges $$\gamma=(\{v_i\}_{i=0}^{n}, \{e_i\}_{i=0}^{n-1})$$ such that $v_i=o(e_i), v_{i+1}=t(e_i)$, and $v_0=a$, $v_n=b$. The number $n$ is called the \textbf{length} of the walk. The set of edges $\{e_i\}_{i=0}^{n-1}$ is allowed to be empty, in which case the length of the walk equals $0$. When $a=b$, we call $\gamma$ a \textbf{closed walk}. A walk of length $0$ is considered as a trivial closed walk. For two walks $\gamma_1=(\{v^1_i\}_{i=0}^{n}, \{e^1_i\}_{i=0}^{n-1})$ and $\gamma_2=(\{v^2_j\}_{j=0}^{m}, \{e^2_i\}_{j=0}^{m-1})$ such that $v^1_n=v^2_0$, we define their multiplication to be a new walk $\gamma_1 \gamma_2:=(\{v^1_0,v^1_1,...,v^1_n,v^2_1,v^2_2,...v^2_m\}, 
\{e^1_0,e^1_1,...,e^1_n, e^2_0,e^2_1,...e^2_m\})$ of length $m+n$.

\begin{definition}\label{defofnbws}
    A \textbf{non-backtracking walk} is a walk $\gamma=(\{v_i\}_{i=0}^{n}, \{e_i\}_{i=0}^{n-1})$ such that $e_{i}\neq \overline{e_{i+1}}$, $0 \leq i\leq n-2$. We regard walks of length $0$ as trivial closed non-backtracking walks. The \textbf{girth} of a graph is the length of its shortest nontrivial closed non-backtracking walk. A \textbf{circuit} is a non-trivial closed non-backtracking walk $\gamma=(\{v_i\}_{i=0}^{n}, \{e_i\}_{i=0}^{n-1})$ such that $e_{0}\neq \overline{e_{n-1}}$. A circuit $\gamma$ is called \textbf{prime} if  $\gamma=\gamma_0^k$ for another circuit $\gamma_0$, then $\gamma=\gamma_0$ and $k=1$. 
\end{definition}
The counting of closed non-backtracking walks and circuits has the following relation, see, for example, Serre \cite{Ser97} and the references therein.
\begin{lemma}\label{NBWandcircle}
Let $G=(V,E)$ be a $(q+1)\text{-}$regular graph. We denote the number of closed non-backtracking walks of length $r$ as $f_r(G)$ and the number of circuits of length $r$ as $c_r(G)$. Then for any $r\geq 1$,
\begin{equation}\label{frcr}
    f_r(G)=c_r(G)+(q-1)\sum_{1\leq i<  r/2}q^{i-1}c_{r-2i}(G).
\end{equation}
\end{lemma}

\begin{remark}
By definition, we always have 
$$c_0(G)\equiv 0 \,\,\, and \,\,\, f_0(G)=|V|.$$
\end{remark}
For any vertex $v\in V$, we denote the number of closed non-backtracking walks of length $r$ from (and to) $v$ as $f_r(v;G)$.

We also need the following definitions.

\begin{definition}\label{primecircuitequi}
    Let $\gamma_0=(\{v_i\}_{i=0}^{n}, \{e_i\}_{i=0}^{n-1})$ be a prime circuit. A prime circuit $\gamma$ is called equivalent to $\gamma_0$ if it differs from $\gamma_0$ by a cyclic rotation, i.e., there exits $0\leq m\leq n$ such that 
$$\gamma=(\{v_{m},v_{m+1},...,v_n,v_0,v_1,...v_{m-1}\}, 
\{e_{m},e_{m+1},...,e_{n-1},e_0,e_1,...e_{m-1}\}).$$
\end{definition}
 We denote the collection of equivalent classes of prime circuits in a graph $G$ as $\mathcal{P}(G)$. The numbers of circuits and prime circuits have the following direct relation.
\begin{lemma}\label{primecir}
Let $G=(V,E)$ be a $(q+1)\text{-}$regular graph. For any $\gamma\in \mathcal{P}(G)$, denote $\ell_\gamma$ the length of any circuit from the equivalent class $\gamma$. There holds
\begin{equation}\label{lpcr}
    c_r(G)=\sum_{\gamma\in \mathcal{P}(G), \ell_\gamma|r}\ell_\gamma.
\end{equation}
\end{lemma}
Let $\pi_{G}(r)$ be the number of the equivalent classes of prime circuits with length $r$. By \eqref{lpcr} and Möbius inversion formula, we have 
\begin{equation}\label{primeandcr}
 \pi_{G}(r)=\frac{1}{r} \sum_{d \mid r} \mu\left(\frac{r}{d}\right) c_{r}(G),
\end{equation}
where $\mu$ is the Möbius function given by
$$
\mu(n)= \begin{cases}1, & n=1, \\ (-1)^k, & n=p_1 \ldots p_k, \text { distinct primes }, \\ 0, & n \text { is divisible by the square of a prime }.\end{cases}
$$

Next, we recall the definition of the non-backtracking matrices.
\begin{definition}
Let $G=(V,E)$ be a $(q+1)$-regular graph with adjacency matrix (operator when $G$ infinite) $A(G)$. The \textbf{non-backtracking matrix of length r}, denoted as $A_r(G)$, is a matrix (resp. operator) indexed by $V \times V$, such that for any vertices $a,b \in V$, $(A_r(G))_{ab}$ equals the number of non-backtracking walks of length $r$ from $a$ to $b$ in $G$. 
\end{definition}

These non-backtracking matrices $\{A_r(G)\}_{r=0}^{\infty}$ are related to the adjacency matrix $A(G)$ as shown in Friedman \cite{Fri91}. We formulate this result below following Serre \cite{Ser97}.
\begin{lemma}\label{green}
    For $|t|<q^{-1/2}$, there holds
\begin{equation}\label{recurrencerelationnbmatrices}
    (1-t^2) (I-A(G)t+qt^2I)^{-1}=\sum_{r=0}^{\infty}A_{r}(G)t^r,
\end{equation}
    where the right hand side converges absolutely with respect to the operator norm.
\end{lemma}

For convenience, we write $A=A(G)$, $A_r=A_{r}(G)$ if no confusion arises. 

\subsection{Normalized spectral measures of regular graphs}

Let $A$ be a bounded self-adjoint operator on a Hilbert space $\mathcal{H}$. For each $\phi \in \mathcal{H}$ with $\Vert \phi\Vert=1$, there exists a unique probability measure $\mu_{A}^{\phi}$ such that
\begin{equation*}
\int_\mathbb{R} P(x) d\mu_{A}^{\phi} (x)=\langle \phi, P(A) \phi \rangle,
\end{equation*}
 holds for any polynomial $P \in \mathbb{R}[x]$.
This measure is called the spectral measure of $A$ at $\phi$. 

Now let $G=(V,E)$ be a $(q+1)$-regular graph. For any vertex $v\in V$, the indicator function $\mathbf{1}_{v}$ is defined such that for $u\in V$,
$$\mathbf{1}_{v}(u)=\left\{
\begin{aligned}
    &1, \ u=v;\\
    &0, \ u\neq v.
\end{aligned} \right.$$ 
We define the normalized spectral measure of $G$ at the vertex $v$ following Bordenave \cite{Bor16}.
\begin{definition}[Normalized spectral measure of regular graph at a vertex]\label{def:spectralmeasureatvertex} Let $G=(V,E)$ be a $(q+1)$-regular graph with adjacency matrix (operator) $A$. For any vertex $v\in V$, the normalized spectral measure $\mu_G^{v}$ of $G$ at $v$ is defined as the spectral measure of the operator $q^{-1/2}A$ on the Hilbert space $\ell^2(V)$ at $\mathbf{1}_v$.

If $G$ is vertex-transitive, then $\mu_G^{v}$ is independent of the choice of $v\in V$. We refer to it as the normalized spectral measure of $G$.
\end{definition}

A regular tree is vertex-transitive. This leads to the following definition.

\begin{definition}\label{kestenmckay}
    The Kesten-McKay distribution $\mu_q$ is the normalized spectral measure of the $(q+1)$-regular tree $\mathbb{T}_q$. Explicitly, $\mu_q$ is the unique probability measure such that for any polynomial $P \in \mathbb{R}[x]$, 
\begin{equation*}
    \int_\mathbb{R} P(x) d\mu_{q}^{o} (x)=\langle \mathbf{1}_{o}, P(q^{-1/2}A(\mathbb{T}_q)) \mathbf{1}_{o} \rangle,
\end{equation*}
where $o$ is any vertex of $\mathbb{T}_q$ and $A(\mathbb{T}_q)$ is the adjacency operator of $\mathbb{T}_q$.
\end{definition}

The explicit form of the Kesten-McKay distribution is given below, see Kesten \cite{Kes59} and McKay \cite{McK81}.
\begin{lemma}
The Kesten-McKay distribution $\mu_q$ for $q=1$ is given by 
\begin{equation*}
d\mu_{1}(x)=\frac{1}{\pi}\frac{1}{\sqrt{4-x^2}}\mathbf{1}_{|x|\leq 2}\ dx.
\end{equation*}
For $q>1$, it is given by 
    \begin{equation*}
    d\mu_{q}(x)=\frac{1}{2\pi}\frac{(q+1)\sqrt{4-x^2}}{ (q^{-1/2}+q^{1/2})^2-x^2}\mathbf{1}_{|x|\leq 2}\ dx.
    \end{equation*} 
\end{lemma}

We denote by $\mu_{\infty}$
the semicircle distribution which is defined as below
$$d\mu_{\infty}(x)=\frac{1}{2\pi}\sqrt{4-x^2}\mathbf{1}_{|x|\leq 2}\ dx.$$
This is motivated by the fact that $\mu_q$ converges to $\mu_{\infty}$ as $q\rightarrow \infty$.

For a finite regular graph $G$, we define its normalized spectral measure as follows.

\begin{definition}[Normalized spectral measure of a finite regular graph]\label{spectralmeasure}
    Let $G=(V,E)$ be a finite $(q+1)$-regular graph with adjacency matrix $A$. The normalized spectral $\mu_G$ of $G$ is defined as 
    $$\mu_G=\frac{1}{|V|}\sum_{v\in V}\mu_G^{v}.$$
\end{definition}   
For finite vertex-transitive graphs, the Definition \ref{def:spectralmeasureatvertex} and Definition \ref{spectralmeasure} of normalized spectral measure coincide. 

It is straightforward to check that the spectral measure of a finite regular graph can be expressed through eigenvalues of its adjacency matrix as below.     

\begin{lemma}
    Let $G=(V,E)$ be a finite $(q+1)$-regular graph and  $\lambda_{|V|}(A) \leq ... \leq \lambda_2(A) \leq \lambda_1(A)$ be the eigenvalues of its adjacency matrix $A$. There holds
    \begin{equation*}
    \mu_G=\frac{1}{|V|}\sum_{1\leq k\leq |V|}\delta_{q^{-1/2} \lambda_k(A)}.
\end{equation*}
\end{lemma}
Notice that for any polynomial $P\in \mathbb{R}[x]$, the normalized spectral measure $\mu_G$ of a finite regular graph satisfies
\begin{equation*}
    \int_\mathbb{R} P d\mu_G= \frac{1}{|V|}\mathrm{Tr}( P(q^{-1/2}A)).
\end{equation*}

\subsection{Chebyshev-type Polynomials}\label{subsection:Chebyshev}

The non-backtracking matrices defined in Section \ref{NBM} are closely related to the following Chebyshev-type polynomials. These polynomials play an important role in our study of spectrum of regular graphs.

\begin{definition}
    Let $r\in \mathbb{N}$ and $q\in \mathbb{N}^+ \cup \{\infty\}$. We define the polynomial $X_{r,q}$ via the the Taylor expansion at $t=0$ of the following  rational function:
    
    \begin{equation}\label{Xrqgeneratingfunction}
    \frac{1-q^{-1}t^2}{1-xt+t^2}=\sum_{r=0}^{\infty}X_{r,q}(x)t^{r}.
    \end{equation}
    We use the convention $\frac{1}{\infty}=0$.
\end{definition}
 Following Serre \cite{Ser02}, we denote 
 \begin{equation}
     Y_r:=X_{r,1}\,\,\text{and}\,\, X_r:=X_{r,\infty}.
 \end{equation}
Notice that these $X_{r,q}$ are polynomials of degree $r$.

\begin{remark}
    Observe that any $x\in \mathbb{C}$ can be written as $x=\xi_x+\xi_x^{-1}$, for some $\xi_x\in \mathbb{C}$ such that $|\xi_x|\leq 1$. For any given $x\in \mathbb{C}$, the right hand side of (\ref{Xrqgeneratingfunction}) converges for all $|t|<|\xi_x|$ and equals the left hand side. The convergence is absolute and uniform on any compact subset $K \subset B(0, |\xi_x|)$.
\end{remark}

For further use, we list the following properties of the Chebyshev-type polynomials. The first two identities can be found in Serre \cite{Ser97}.
\begin{lemma}[Properties of the Chebyshev-type polynomials]\label{propertyofcheby}
Let $\{X_{r,q}\}$ be the polynomials defined above. The following identities hold:
\begin{itemize}
    \item [(i)]
    For $r\in \mathbb{N}$,
    \begin{equation}\label{chebyinversehao}
        X_{r,q}=X_r-q^{-1}X_{r-2} \,\,\,\,\,and \,\,\,\,\,X_r=\sum_{0\leq k\leq r/2}q^{-k}X_{r-2k,q}.
    \end{equation}
    
    \item [(ii)]
    For $r\in \mathbb{N}$ and $z\in \mathbb{C}\setminus \{0\}$, 
    \begin{equation}\label{chebyshevexplicit}
        X_{r}(z+z^{-1})=\frac{z^{r+1}-z^{-r-1}}{z-z^{-1}} \,\,\,\,\,and \,\,\,\,\,Y_{r}(z+z^{-1})=z^{r}+z^{-r}.
    \end{equation}

    \item [(iii)] For $z\in \mathbb{C}$,
    \begin{equation} \label{eq:Yrderivative}
        Y'_{r}(z)=r X_{r-1}(z).
    \end{equation}
\end{itemize}
In the above, we use the convention $X_r\equiv 0$ for $r\in \mathbb{Z}$, $r<0$.
\end{lemma}

\begin{remark}
  For $r\geq 1$, $Y_r$ and $X_r$ are related to the \textbf{Chebyshev polynomials of the first and second kind} $T_r$ and $U_r$, respectively, by a change of variables:
   $$Y_r(x)=2T_r(x/2),\,\,\, X_r(x)=U_r(x/2).$$
  The polynomials $Y_r$ and $X_r$ are also referred to as \emph{Vieta–Lucas polynomials and Vieta–Fibonacci polynomials}, respectively. 
\end{remark}

Comparing (\ref{recurrencerelationnbmatrices}) with (\ref{Xrqgeneratingfunction}), we obtain the following relation established by Friedman \cite[Lemma 3.3]{Fri91}.

\begin{lemma}\label{ArandXrq} 
For a $(q+1)$-regular graph $G$ where $q\in \mathbb{N}^+$, we have
\begin{equation}\label{Serrechebyshevmoment}
    A_r(G)=q^{r/2}X_{r,q}(q^{-1/2}A(G)),\ r\in \mathbb{N}.
\end{equation}
\end{lemma}

The following three lemmas are consequences of Lemma \ref{ArandXrq} applying to regular trees and finite regular graphs, respectively. See Serre \cite{Ser97} for a proof.
\begin{lemma}\label{lemma:orthogonal}
    For $q\in \mathbb{N} \cup \{\infty\}$, 
    the polynomials $\{X_{r,q}\}_{r=0}^{\infty}$ 
satisfy 
\begin{equation}\label{orthogonalrelation}
\int_\mathbb{R} X_{n,q} X_{m,q}d\mu_{q}=\left\{
\begin{aligned}
    &0, \ m\neq n;\\
    &1, \ m=n=0;\\
    &1+q^{-1}, \ m=n\neq 0.
\end{aligned} \right.
\end{equation}
and form a complete orthogonal basis of $L^2(\mathbb{R},\mu_{q})$. Again we use the convention $\frac{1}{\infty}=0$ and recall that $\mu_{\infty}$ is the semicircle distribution.
\end{lemma}

\begin{lemma}\label{Xrandmuq}
For any $r \in \mathbb{N}$, and $q \in \mathbb{N}^{+}$,
$$
\int_{\mathbb{R}} X_r d \mu_q= \begin{cases}q^{-r / 2}, & \text { if } r \text { is even; } \\ 0, & \text { if } r \text { is odd, }\end{cases}
$$
\end{lemma}

\begin{lemma}\label{polytrace}
    Let $G=(V,E)$ be a finite $(q+1)\text{-}$regular graph with normalized spectral measure $\mu_G$. Then for $r \in \mathbb{N}$,
    \begin{equation}\label{1stchebgeo}
    \int_{\mathbb{R}} X_{r,q}(x)d\mu_G(x)=q^{-r/2}\frac{f_r(G)}{|V|},
\end{equation}
\begin{equation}\label{2ndchebgeo}
\int_{\mathbb{R}} X_{r}(x)d\mu_G(x)=q^{-r/2}\frac{1}{|V|}\sum_{0\leq k\leq r/2} f_{r-2k}(G).
\end{equation}
\end{lemma}

Another result we need is the following prime theorem for regular graphs, see for example, Serre \cite{Ser97}, Rangarajan \cite{RB18}, Horton, Newland, and Terras \cite{HNT06}. For readers' convenience, we present a proof here. 

\begin{theorem}[Prime theorem on regular graph]\label{primethm}
For any $r \in \mathbb{N}$, we have
\begin{equation}\label{circuitnumber}
    \int_{\mathbb{R}} Y_{r}(x)d\mu_G(x)=\int_{\mathbb{R}} Y_{r}(x)d\mu_q(x)+q^{-r/2}\frac{c_r(G)}{|V|}.
\end{equation}
\end{theorem}

\begin{proof}
By Lemmas \ref{polytrace} and \ref{propertyofcheby}, we have
\begin{equation*}\label{eq1}
\begin{aligned}
\int_{\mathbb{R}} Y_{r}(x)d\mu_G(x)&=\int_{\mathbb{R}} (X_{r}(x)-X_{r-2}(x))d\mu_G(x)\\
&=\frac{q^{-r/2}}{|V|}\sum_{0\leq k\leq r/2}f_{r-2k}(G)-\frac{q^{-r/2+1}}{|V|}\sum_{0\leq k\leq r/2-1}f_{r-2-2k}(G)\\
&=\frac{q^{-r/2}}{|V|}\left(f_{r}(G)-(q-1)\sum_{1\leq k\leq r/2}f_{r-2k}(G)\right).
\end{aligned}
\end{equation*}
By Lemma \ref{NBWandcircle}, we derive $$f_{r-2k}(G)=c_{r-2k}(G)+(q-1)\sum_{k+1\leq i<r/2}q^{i-k-1}c_{r-2i}(G).$$
Taking summation over $1\leq k<r/2$ yields
\begin{equation*}\label{eq2}
\begin{aligned}
\sum_{1\leq k< r/2}f_{r-2k}(G)&=\sum_{1\leq k<r/2}c_{r-2k}(G)+(q-1)\sum_{1\leq k<r/2}\sum_{k+1\leq i<r/2}q^{i-k-1}c_{r-2i}\\
&=\sum_{1\leq k<r/2}c_{r-2k}(G)+(q-1)\sum_{2\leq i<r/2}\sum_{1\leq k<i-1}q^{i-k-1}c_{r-2i}\\
&=\sum_{1\leq k<r/2}q^{k-1}c_{r-2k}(G).
\end{aligned}
\end{equation*}
By Lemma \ref{Xrandmuq}, we obtain 
\begin{equation*}\label{eq3}
\int_{\mathbb{R}} Y_r d \mu_q=\int_{\mathbb{R}} (X_{r}-X_{r-2}) d \mu_q=\begin{cases} -q^{-r / 2}(q-1), & \text { if } r \text { is even; } \\ 0, & \text { if } r \text { is odd. }\end{cases}
\end{equation*}
Combining the above equations, and Lemma \ref{NBWandcircle}, we obtain
\begin{equation*}
\begin{aligned}
&\int_{\mathbb{R}} Y_{r}(x)d\mu_G(x)=\frac{q^{-r/2}}{|V|}\left(f_{r}(G)-(q-1)\sum_{1\leq k\leq r/2}f_{r-2k}(G)\right)\\
&=\frac{q^{-r/2}}{|V|}\left(f_{r}(G)-(q-1)\sum_{1\leq k<r/2}f_{r-2k}(G)\right)+\int_{\mathbb{R}}Y_{r}(x)d\mu_{q}(x)\\
&=\int_{\mathbb{R}}Y_{r}(x)d\mu_{q}(x)+\frac{q^{-r/2}}{|V|}\left(f_{r}(G)-(q-1)\sum_{1\leq k<r/2}q^{k-1}c_{r-2k}(G)\right)\\
&=\int_{\mathbb{R}}Y_{r}(x)d\mu_{q}(x)+q^{-r/2}\frac{c_{r}(G)}{|V|}.
\end{aligned}
\end{equation*}
This completes the proof.
\end{proof}
The equivalence classes of prime circuits are considered analogous to prime numbers. Thus, we are concerned with the asymptotic behavior of $\pi_{G}(r)$ as $r$ tends to infinity. Theorem \ref{primethm} and \eqref{primeandcr} show that $\pi_{G}(r)\sim q^{r}/r$ for non-bipartite $G$, and $\pi_{G}(2r)\sim q^{2r}/r$ for bipartite $G$ as $r \to \infty$ (we note that $\pi_{G}(2r+1)=0$ for all $r \in \mathbb{N}$). Furthermore, the remainder term of $\pi_{G}(r)$ depends on the spectral gap, which is the distance between the non-trivial eigenvalues $\{\lambda_{k}: k\geq 2\}$ and $\{\pm \left(q^{1/2}+q^{-1/2}\right)\}$. This implies that the number of prime circuits on Ramanujan graphs has remainder terms of $\mathcal{O}(q^{r/2})$.

\subsection{The Joukowsky transform}
Recall that the Joukowsky transform is a conformal map $\mathcal{J}: \mathbb{C} \setminus \{0\} \rightarrow \mathbb{C}$ such that 
$$\mathcal{J}(z)=z+z^{-1}.$$
For simplicity, we make the following convention.
\begin{definition}\label{interiorellipse}
    For $\rho>0, \rho\neq 1$, we define $\Omega(\rho)$ to be the following open set 
    $$\Omega(\rho):=\left\{w\in \mathbb{C}:\frac{(\mathrm{Re}\ w)^2}{(\rho+\rho^{-1})^2}+\frac{(\mathrm{Im}\ w)^2}{(\rho-\rho^{-1})^2}<1\right\}.$$
\end{definition}

It's direct to see the map $\mathcal{J}$ is onto, and maps the unit circle to the interval $[-2,2]$.
Moreover, the following holds.
\begin{lemma}
    Let $\rho>0, \rho \neq 1$. The Joukowsky transform maps the circle $\partial B(0,\rho)=\{|z|=\rho\}$ bijectively to the ellipse $\partial\Omega(\rho)$.
\end{lemma}

\section{Resolvent and Stieltjes transform}

For any self-adjoint operator $A$, the resolvent is defined by
$$R(A; z):=(zI-A)^{-1}$$ for any $z\notin \mathrm{spec}(A)$. For the adjacency matrix or operator $A(G)$ of a $(q+1)$-regular graph $G$, the identity
$$(1-t^2) (I-A(G)t+qt^2I)^{-1}=\sum_{r=0}^{\infty}A_{r}(G)t^r$$
from Lemma \ref{green} in fact provides a decomposition of the resolvent of $A(G)$ in terms of the non-backtracking matrices $A_r(G)$ under a change of variables $z=t+{t}^{-1}$. 

In case that $G$ is finite, we can take traces of both sides to derive a decomposition of the Stieltjes transform of the normalized spectral measure in terms of number of non-backtracking walks.

Let $\mu$ be a probability measure on the real line. The Stieltjes transform $S_{\mu}$ of $\mu$ is defined as follows: 
\begin{equation*}
    S_{\mu}(z):=\int_{\mathbb{R}}\frac{1}{x-z}d\mu(x),\,\,\text{for any}\,\,z\notin\mathrm{supp}\ \mu.
\end{equation*}
For practical reasons, we define the following modified Stieltjes transform via the generating function of the Chebyshev-type polynomials.

\begin{definition}
    We define the modified Stieltjes transform $\widetilde{S}_{\mu,q}$ of a probability measure $\mu$ on $\mathbb{R}$ by
\begin{equation*}
    \widetilde{S}_{\mu,q}(t):=\int_{\mathbb{R}}\frac{1-q^{-1}t^2}{1-xt+t^2}d\mu(x),
\end{equation*}
for all $t\in \mathbb{C}$ such that the generating function is integrable with respect to $\mu$.
\end{definition}
\begin{remark}\label{rmk:modifiedStieltjes}
    The original and modified Stieltjes transforms of a probability measure $\mu$ are related by 
    \begin{equation*}
        \widetilde{S}_{\mu,q}(t)=(t^{-1}-q^{-1}t)S_{\mu}(t+t^{-1}),
    \end{equation*}
    whenever both transforms are well-defined.
\end{remark}

The modified Stieltjes transform of the Kesten-McKay distribution is easy to calculate due to the orthogonal relation (\ref{orthogonalrelation}).

\begin{proposition}\label{prop:modifiedStieltjes}
The modified Stieltjes transform of Kesten-McKay distribution $\mu_q$ for $|t|<1$ is given by
    \begin{equation*}
    \widetilde{S}_{\mu_q,q}(t)\equiv1.
\end{equation*}
\end{proposition}
\begin{proof}
For any $|t|< 1$, the generating function is integrable as a function of $x$ with respect to $\mu_q$. Observe that  (\ref{Xrqgeneratingfunction}) holds for for any $|x|\leq 2$ and $|t|< 1$. Since $\mathrm{supp}\ \mu_q=[-2,2]$, we derive
\begin{equation*}
\begin{aligned}
    \int_{\mathbb{R}}\frac{1-q^{-1}t^2}{1-xt+t^2}d\mu_q(x)&=\int_{\mathbb{R}}\sum_{r=0}^{\infty}X_{r,q}(x)t^{r}d\mu_q(x)\\
    &=\sum_{r=0}^{\infty}t^{r}\int_{\mathbb{R}}X_{r,q}(x)d\mu_q(x)=1.
\end{aligned}
\end{equation*}
Thus the proof is completed.
\end{proof}

Since for any $z\notin [-2,2]$, there exits a unique $t$, $|t|<1$ such that $z=t+1/t$, we have the following consequence of Proposition \ref{prop:modifiedStieltjes} and Remark \ref{rmk:modifiedStieltjes}.

\begin{corollary}
The Stieltjes transform of Kesten-McKay distribution $\mu_q$ for $z\notin [-2,2]$ is given by
    \begin{equation*}
    \int_{\mathbb{R}}\frac{1}{x-z}d\mu_q(x)=\frac{t}{q^{-1}t^2-1},\,\,\,\text{where}\,\,z=t+\frac{1}{t},\,\,\,|t|<1.
\end{equation*}
\end{corollary}

Let $G$ be a finite regular graph with normalized spectral measure $\mu_G$. By definition, the modified Stieltjes transform of $\mu_G$ is holomorphic near the origin. In fact, its Taylor expansion at $t=0$ counts the number of non-backtracking walks of $G$.

\begin{proposition}\label{stieltjesq}
    Let $G=(V,E)$ be a finite $(q+1)$-regular with normalized spectral measure $\mu_G$. For $|t|<q^{-1/2}$, 
    \begin{equation*}
        \widetilde{S}_{\mu_G,q}(t)=\sum_{r=0}^{\infty}q^{-r/2}\frac{f_r(G)}{|V|}t^r.
    \end{equation*}
\end{proposition}
\begin{proof}
    Note that $\mathrm{supp}\ \mu_G \subset [-q^{1/2}-q^{-1/2}, q^{1/2}+q^{-1/2}]$. Then for any $|t|< q^{-1/2}$, the generating function is integrable as a function of $x$ with respect to $\mu_G $.  Moreover, the identity (\ref{Xrqgeneratingfunction}) holds for $|t|< q^{1/2}$. Thus \begin{equation*}
\begin{aligned}
    \int_{\mathbb{R}}\frac{1-q^{-1}t^2}{1-xt+t^2}d\mu_G(x)&=\int_{\mathbb{R}}\sum_{r=0}^{\infty}X_{r,q}(x)t^{r}d\mu_G(x)\\
    &=\sum_{r=0}^{\infty}t^{r}\int_{\mathbb{R}}X_{r,q}(x)d\mu_G(x)=\sum_{r=0}^{\infty}t^{r}q^{r/2}\frac{f_r(G)}{|V|}.
\end{aligned}
\end{equation*}
Here the last equality is due to (\ref{1stchebgeo}). This finishes the proof.
\end{proof}

In the same manner, we derive the following decomposition of transform $\widetilde{S}_{\mu_G,1}(t)$ in terms of number of circuits.

\begin{proposition}\label{Iharababy}
    Let $G=(V,E)$ be a finite $(q+1)$-regular with normalized spectral measure $\mu_G$. For $|t|<q^{-1/2}$, 
    \begin{equation*}
        \widetilde{S}_{\mu_G,1}(t)=\frac{1-t^2}{1-q^{-1}t^2}+\sum_{r=1}^{\infty}q^{-r/2}\frac{c_r(G)}{|V|}t^r.
    \end{equation*}
\end{proposition}
\begin{proof}
    By Theorem \ref{primethm}, for $|t|< q^{1/2}$, we have
\begin{equation*}
\begin{aligned}
    &\widetilde{S}_{\mu_G,1}(t)-\frac{1-t^2}{1-q^{-1}t^2}=\int_{\mathbb{R}}\frac{1-t^2}{1-xt+t^2}d\mu_G(x)-\frac{1-t^2}{1-q^{-1}t^2}\\
    &=\int_{\mathbb{R}}\frac{1-t^2}{1-xt+t^2}d\mu_G(x)-\int_{\mathbb{R}}\frac{1-t^2}{1-xt+t^2}d\mu_q(x)\\
    &=\int_{\mathbb{R}}\sum_{r=0}^{\infty}Y_{r}(x)t^{r}d\mu_G(x)-\int_{\mathbb{R}}\sum_{r=0}^{\infty}Y_{r}(x)t^{r}d\mu_q(x)\\
    &=\sum_{r=1}^{\infty}t^{r}\left(\int_{\mathbb{R}}Y_{r}(x)d\mu_G(x)-\int_{\mathbb{R}}Y_{r}(x)d\mu_q(x)\right)=\sum_{r=1}^{\infty}t^{r}q^{r/2}\frac{c_r(G)}{|V|}.
\end{aligned}
\end{equation*}
Here we use the fact that
$$
\begin{aligned}
    &\int_{\mathbb{R}} \frac{1-t^2}{1-xt+t^2}d\mu_q(x)=\frac{1-t^2}{1-q^{-1}t^2}\int_{\mathbb{R}} \frac{1-q^{-1}t^2}{1-xt+t^2}d\mu_q(x)=\frac{1-t^2}{1-q^{-1}t^2}.
\end{aligned}
$$
\end{proof}

Proposition \ref{stieltjesq} and Proposition \ref{Iharababy} can be considered as trace formulas for generating functions of Chebyshev-type polynomials. To establish the trace formula for general functions, we need the technical results presented in the next section.

\section{Functional calculus of adjacency matrix}

We prove the following decomposition of a holomorphic function along Chebyshev-type polynomials. It plays a fundamental role in our proof of the discrete trace and kernel formulas and is of its own interest as a result of complex analysis. We delay the proof to Section \ref{section:proof}. Recall the definition of open set $\Omega(\rho), \rho>0, \rho\neq 1$ 
 \begin{equation*}
        \Omega(\rho):=
    \left\{z\in \mathbb{C}:\frac{(\mathrm{Re}\ z)^2}{(\rho+\rho^{-1})^2}+\frac{(\mathrm{Im}\ z)^2}{(\rho-\rho^{-1})^2}
    < 1 \right\}.
    \end{equation*}

\begin{theorem}\label{master}
    Suppose that $h$ is holomorphic on $\Omega(\rho), \rho>0, \rho\neq 1$.
    Then for all $q\in \mathbb{N}\cup \{\infty\}$ and $z \in \Omega(\rho)$, we have
    \begin{equation}\label{masterlemma}
        h(z)=\sum_{r=0}^{\infty}a_{r,q}(h)X_{r,q}(z),
    \end{equation}
    where
    \begin{equation}\label{crq}
    \begin{aligned}
        a_{r,q}(h)=\frac{1}{2\pi i}\oint_{\partial B(0,1)} h(\xi+\xi^{-1})\xi^{r-1}\frac{1-\xi^2}{1-q^{-1}\xi^2}d\xi.
    \end{aligned}
    \end{equation}
    Moreover, the right hand side of (\ref{masterlemma}) converges absolutely and uniformly on any compact subset of $\Omega(\rho)$.
\end{theorem}

The following proposition is on the relationship of the coefficients $a_{r,q}(h)$ in Theorem \ref{master}.

\begin{proposition}\label{xrqhuxiangzhuanhuan}
    Let $h$ and $\{a_{r,q}(h)\}$ be as in Theorem \ref{master}. For $r\in \mathbb{N}$,  we have
\begin{equation}{\label{cheby1tocheb2}}
   a_{r,\infty}(h)=a_{r,1}(h)-a_{r+2,1}(h),
\end{equation}
and for $q\in \mathbb{N}^+$,
\begin{equation}\label{chebinftytochebq}
   a_{r,q}(h)=\sum_{k=0}^{\infty}
 q^{-k}a_{r+2k,\infty}(h).
\end{equation}
\end{proposition}

We delay the proof to Section \ref{section:proof}. 

A direct corollary of Theorem \ref{master} is the following functional calculus formula for regular graphs which appears as Theorem \ref{introexpansion} in the introduction. Note that the graphs here can be infinite.
\begin{theorem}\label{kernel}
    Let $G=(V,E)$ be a $(q+1)$-regular graph with normalized spectral measure $\mu_G$.
    Suppose $h$  is holomorphic on $\Omega(\rho)$ for $\rho>q^{1/2}$.
    The following formula holds
    \begin{equation}\label{eq:functionalcalculus}
        h(q^{-1/2}A)=\int_{\mathbb{R}} h(x)d\mu_q(x)\mathrm{I}+\sum_{r=1}^{\infty}q^{-r/2}a_{r,q}(h) A_r,
    \end{equation}
    where for $r\geq 1$,
    \begin{equation*}
        a_{r,q}(h)=(1+q^{-1})^{-1}\int_{-2}^{2}h(x)X_{r,q}(x)d\mu_q(x).
    \end{equation*}
The right hand side of (\ref{eq:functionalcalculus}) converges uniformly with respect to the operator norm.
\end{theorem}

\begin{proof}
    Notice that $\Vert q^{-1/2}A\Vert\leq q^{1/2}+q^{-1/2}$. Since $h$ is holomorphic on $\Omega(\rho)$, $\rho>q^{1/2}$, we derive by Theorem $\ref{master}$, \begin{equation*}
        h(z)=\sum_{r=0}^{\infty}a_{r,q}(h)X_{r,q}(z),
    \end{equation*}
    which converges uniformly on $\overline{\Omega(q^{1/2})}$. By the orthogonal relation (\ref{orthogonalrelation}), we have for $q>1$,
    \begin{equation*}
        a_{r,q}(h)=(1+q^{-1})^{-1}\int_{-2}^{2}h(x)X_{r,q}(x)d\mu_q(x).
    \end{equation*}
    Thus 
    \begin{equation*}
        h(q^{-1/2}A)=\sum_{r=0}^{\infty}a_{r,q}(h)X_{r,q}(q^{-1/2}A)=\sum_{r=0}^{\infty}q^{-r/2}a_{r,q}(h)A_r.
    \end{equation*}
    This completes the proof.
\end{proof}

\section{Discrete pre-trace and trace formulas}

Notice that for a $(q+1)$-regular graph $G=(V,E)$, the support of its normalized spectral measure $\mu_G^v$ at a vertex $v$, is contained in $[-q^{1/2}-q^{-1/2}, q^{1/2}+q^{-1/2}]$. In this section, we apply Theorem \ref{master} and Theorem \ref{kernel} to obtain the pre-trace and trace formulas for regular graphs. 

\subsection{Pre-trace formula}
Recall that we denote the number of closed non-backtracking walks of length $r$ from (and to) $v$ by $f_r(v;G)$. The following is Theorem \ref{intropretrace} in the introduction.
\begin{theorem}[Discrete pre-trace formula]
    Let $G=(V,E)$ be a $(q+1)$-regular graph (not necessarily finite). Suppose 
    $h$ is holomorphic on $\Omega(\rho)$ for $\rho>q^{1/2}$. For any $v \in V$, the following pre-trace formula holds
    \begin{equation*}
\int_{\mathbb{R}} h(x) d\mu_{G}^{v}(x)=
\int_{\mathbb{R}} h(x) d\mu_q(x)+\sum_{r=1}^{\infty}q^{-r/2}a_{r,q}(h)f_r(v;G),
\end{equation*}
where $$a_{r,q}(h)=(1+q^{-1})^{-1}\int_{-2}^{2}h(x)X_{r,q}(x)d\mu_q(x).$$
\end{theorem}

\begin{proof}
    By Theorem \ref{kernel}, we have
    \begin{equation*}
        h(q^{-1/2}A)=\int_{\mathbb{R}} h(x)d\mu_q(x)\mathrm{I}+\sum_{r=1}^{\infty}q^{-r/2}a_{r,q}(h) A_r.
    \end{equation*}
    Hence $$
    \begin{aligned}
    \int_{\mathbb{R}} h(x) d\mu_{G}^{v}(x)&=\langle \mathrm{1}_v, h(q^{-1/2}A) \mathrm{1}_v \rangle\\
    &=\int_{\mathbb{R}} h(x) d\mu_q(x)+\sum_{r=1}^{\infty}q^{-r/2}a_{r,q}(h)\langle \mathrm{1}_v, A_r \mathrm{1}_v \rangle\\
    &=\int_{\mathbb{R}} h(x) d\mu_q(x)+\sum_{r=1}^{\infty}q^{-r/2}a_{r,q}(h) f_r(v;G).
    \end{aligned}
    $$
    We complete the proof.
\end{proof}

\subsection{Trace formula}
Next we show the trace formula for finite regular graphs. Notice that this is not a direct corollary from the above pre-trace formula. Recall that we denote by 
\[\lambda_{|V|}(A) \leq ... \leq \lambda_2(A) \leq \lambda_1(A)\]
the eigenvalues of adjacency matrix $A(G)$ of a finite graph $G=(V,E)$.
The following theorem is Theorem \ref{introdiscretetracecircuit} appearing in the introduction.
\begin{theorem}[Discrete trace formula, circuit version]\label{discretetracecircuit}
    Let $G=(V,E)$ be a finite $(q+1)$-regular graph. Suppose $h$ is holomorphic on $\Omega(\rho)$ for $\rho>q^{1/2}$. Then
    the following discrete trace formula holds
    \begin{equation}\label{discretetrace}
\sum_{k=1}^{|V|} h(q^{-1/2}\lambda_k(A))=
\int_{\mathbb{R}} h(x) d\mu_q(x)|V|+\sum_{r=1}^{\infty}q^{-r/2}a_{r,1}(h)c_r(G),
\end{equation}
    where 
    \begin{equation}\label{arhtrace}
        a_{r,1}(h)=\frac{1}{2\pi i}\oint_{\partial B(0,1)} h(\xi+\xi^{-1})\xi^{r-1}d\xi.
    \end{equation}
The right hand side of (\ref{discretetrace}) converges absolutely. 
\end{theorem}

\begin{proof}
By Theorem \ref{master}, we have
$$h(z)=\sum_{r=0}^{\infty}a_{r,1}(h)Y_r(z),$$
where the right hand side converges absolutely and uniformly on $\overline{\Omega(q^{1/2})}$. Especially, 
\begin{equation}\label{absolute}
    \sum_{r=0}^{\infty}|a_{r,1}(h)|(q^{r/2}+q^{-r/2})<\infty.
\end{equation}
By Theorem \ref{primethm}, we derive
$$
\begin{aligned}
&\int_{\mathbb{R}} h(x)d\mu_G(x)-\int_{\mathbb{R}} h(x)d\mu_q(x)\\
&=\int_{\mathbb{R}} \sum_{r=0}^{\infty}a_{r,1}(h)Y_r(x) d\mu_G(x)-\int_{\mathbb{R}} \sum_{r=0}^{\infty}a_{r,1}(h)Y_r(x) d\mu_q(x)\\
&=\sum_{r=1}^{\infty}a_{r,1}(h) \left(
\int_{\mathbb{R}} Y_r(x) d\mu_G(x)-\int_{\mathbb{R}} Y_r(x) d\mu_q(x)
\right)\\
&=\sum_{r=1}^{\infty}a_{r,1}(h) q^{-r/2}\frac{c_r(G)}{|V|}.
\end{aligned}
$$
By definition, we have
$$\int_{\mathbb{R}} h(x)d\mu_G(x)=\frac{1}{|V|}\sum_{k=1}^{|V|} h(q^{-1/2}\lambda_k(A)).$$ Thus we prove (\ref{discretetrace}). The absolute convergence is due to the fact $c_r(G)\leq (q+1)q^{r-1}$ and (\ref{absolute}).
\end{proof}

We reformulate our trace formula in the following form, which resembles perfectly the Selberg trace formula for compact hyperbolic surfaces in appearance. 

\begin{theorem}[Discrete trace formula, prime version]\label{Discretetraceformulaprime version}
    Let $G=(V,E)$ be a finite $(q+1)$-regular graph. Suppose $h$ is holomorphic on $\Omega(\rho)$ for $\rho>q^{1/2}$. Then
    the following discrete trace formula holds
    \begin{equation*}
    \begin{aligned}
\sum_{k=1}^{|V|} h(q^{-1/2}\lambda_k(A))=
\frac{|V|}{2\pi}\int_{-2}^{2} h(x) \frac{(q+1)\sqrt{4-x^2}}{ (q^{-1/2}+q^{1/2})^2-x^2}dx+\sum_{\gamma \in \mathcal{P}(G)} \sum_{n=1}^{\infty}\frac{\ell_\gamma \hat{h}(n\ell_\gamma)}{q^{n\ell_\gamma/2}}
,
\end{aligned}
\end{equation*}
where $\hat{h}(n)$, $n\in \mathbb{N}^+$ is the $n$-th coefficient of the Chebyshev series of $h$, i.e, 
\begin{equation*}
    \hat{h}(n)=\frac{1}{\pi}\int_{0}^{\pi}h(2\cos\theta)\cos(n\theta)d\theta.
\end{equation*}
\end{theorem}
\begin{proof}[Proof of Theorem \ref{Discretetraceformulaprime version}]
    Notice that $a_{r,1}(h)=\hat{h}(r)$. By Lemma \ref{primecir}, we have \[c_r(G)=\sum\limits_{\gamma\in \mathcal{P}(G),\ell_\gamma|r}\ell_\gamma.\] Then there holds
    $$\begin{aligned}
        \sum_{r=1}^{\infty}q^{-r/2}a_{r,1}(h)c_r(G)&=\sum_{r=1}^{\infty}q^{-r/2} \hat{h}(r) \sum_{\gamma,\ell_\gamma|r}\ell_\gamma\\
       &=\sum_{\gamma \in \mathcal{P}(G)}\sum_{n=1}^{\infty}\ell_\gamma \hat{h}(n\ell_\gamma) q^{-n\ell_\gamma/2}.
    \end{aligned}$$
    Hence the proof is completed.
\end{proof}
Conversely, if we assume that $\{a_{n}\}_{n=0}^{\infty}$ satisfies 
\begin{equation}\label{convergent}
    \sum_{n=0}^{\infty}|a_{n}|q^{n/2}<+\infty,
\end{equation}
then $h(z)=\sum_{n=0}^{\infty}a_{n}Y_{n}(z)$ is uniformly convergent on $\overline{\Omega(q^{1/2})}$. Thus the trace formula in Theorem \ref{Discretetraceformulaprime version} holds for such $h$, which is consistent with a special case of Ahumada's trace formula \cite{Ahu87} with trivial representation of the fundamental group of $G$. Our assumption for the test function $h$ gives a criterion for the convergent condition \eqref{convergent}.
\begin{remark}\label{rmk:comparison}
As a comparison, the \textbf{Selberg trace formula} for compact hyperbolic surfaces \cite{Sel56,Bus92} states that for any admissible even function $h$,
\begin{equation}\label{Seltrace}
\sum_{j=0}^{\infty} h\left(r_j\right)=\frac{\operatorname{Area}(X)}{4 \pi} \int_{-\infty}^{+\infty} h(r) r \tanh (\pi r) d r+\sum_{\gamma \in \mathcal{P}(M)} \sum_{n=1}^{+\infty} \frac{\ell_\gamma \hat{h}\left(n \ell_\gamma\right)}{2 \sinh \left(n \ell_\gamma / 2\right)},
\end{equation}
where $r_{j}$ satisfy that $\frac{1}{4}+r_{j}^2=\lambda_{j}$, where $\lambda_{j}$ is the $j$-th eigenvalue of Laplacian, $\mathcal{P}(M)$ is the set of all oriented prime closed
geodesics, and $\hat{h}$ is the Fourier transform of $h$, i.e. $$\hat{h}(u)=\mathcal{F}(h)(u)=\int_{-\infty}^{+\infty} e^{iru} h(r) d r.$$
\end{remark}

\subsection{Relation to harmonic analysis on regular trees}
For a finite $(q+1)$-regular graph $X$, it is a quotient space $X = \Gamma \setminus \mathbb{T}_{q}$, where $\Gamma$ is the fundamental group of $X$. Terras and Wallace \cite{TW03} present the following discrete trace formula for a finite $(q+1)$-regular graph $X$.
\begin{theorem}[Terras-Wallace \cite{TW03}]\label{terrastrace}
Suppose that $f: \mathbb{T}_{q} \rightarrow \mathbb{R}$ has finite support and is invariant under rotation about the root $o$ of $\mathbb{T}_{q}$. Let $\mathcal{P}_{\Gamma}(X)$ be the set of prime hyperbolic conjugacy classes in $\Gamma$. Then
\begin{equation}\label{eq:TerrasWallace}
\sum_{i=1}^{|X|} \widehat{f}\left(s_i\right)=f(o)|X|+\sum_{\{\rho\} \in \mathcal{P}_{\Gamma}} v(\rho) \sum_{k \geq 1} \mathrm{H} f(k v(\rho)) .
\end{equation}
\end{theorem}
Let us explain the notations in (\ref{eq:TerrasWallace}): 
$s_i$ satisfy that $q^{s_{i}}+q^{1-s_{i}}=\lambda_{i}$ is the $i$-th eigenvalue of adjacency operator $A$ on $X$. $\widehat{f}$ is the spherical transform of $f$ satisfying (see \cite[Lemma 3.3]{TW03})
\begin{equation}\label{spherical}
\widehat{f}(s)=\sum_{n \in \mathbb{Z}} \mathrm{H} f(n) q^{|n| / 2} z^n, \quad z=q^{s-1/2},
\end{equation}
where the horocycle transform $\mathrm{H} f$ of $f$ is defined by
\begin{equation}\label{horo}
\mathrm{H} f(n)=f(|n|)+(q-1) \sum_{j \geq 1} q^{j-1} f(|n|+2 j), \quad \text { for } n \in \mathbb{Z}.
\end{equation}
\noindent Inversely, there holds 
\begin{equation}\label{inversehoro}
f(|n|)=\mathrm{H} f(|n|)-(q-1) \sum_{j \geq 1}(\mathrm{H} f)(|n|+2 j) .
\end{equation}
The number $v(\rho)$ is the integer giving the size of the shift by $\rho$ along its fixed doubly non-backtracking walk $\gamma=\{\cdots,x_{n},\cdots,x_{-1},x_{0},x_{1},\cdots,x_{n}\cdots\}$ on $\mathbb{T}_{q}$. $\gamma$ can be regarded as a geodesic on the tree. The prime hyperbolic conjugacy classes are in one-to-one correspondence with equivalent classes of prime circuits (see Stark’s article in \cite[pages 601-615]{HFGO12}). For $\{\rho\} \in \mathcal{P}_{\Gamma}$ corresponding to $\gamma\in \mathcal{P}(X)$, there holds $v(\rho)=\ell(\gamma)$. Next, we provide the relation between two trace formulas in Theorem \ref{Discretetraceformulaprime version} and Theorem \ref{terrastrace}. Denote $l^2_S(\mathbb{T}_q, o)\subset l^2(\mathbb{T}_q)$ such that for $f \in l^2_S(\mathbb{T}_q, o)$ and $v_1, v_2 \in V_{\mathbb{T}_q}$, $f(v_1)=f(v_2)$ if $d(v_1, o)=d(v_2, o)$.

\begin{theorem}\label{thm:relation}
Let $\mathbb{T}_q$ be a $(q+1)$-regular tree with a root $o$ and $f \in l^2_S(\mathbb{T}_q, o)$ with finite support. Suppose there is a function $h$ satisfying
\[h\left(q^{s-1/2}+q^{-s+1/2}\right)=\widehat{f}(s)\]
Then we have the following inverse formulas for $n\geq 0$
\begin{equation}\label{relationbetweentwotrace}
\begin{gathered}
\mathrm{H}f(n)=a_{n,1}(h)q^{-|n|/2}, \\
f(n)=a_{n,q}(h)q^{-|n|/2},
\end{gathered}
\end{equation}
where $a_{n,1}(h)$ and $a_{n,q}(h)$ are $n$-th coefficient of $h=\sum_{n=0}^{\infty}a_{n,1}(h)Y_{n}$ and $h=\sum_{n=0}^{\infty}a_{n,q}(h)X_{n,q}$, respectively. 
\end{theorem}
Theorem \ref{thm:relation} implies that the two trace formulas in Theorem \eqref{terrastrace} and Theorem \eqref{Discretetraceformulaprime version} are consistent.

\begin{proof}
Let $z=z(s)=q^{s-1/2}$.
By the formula of spherical transform \eqref{spherical}, we have
\begin{equation}
\begin{aligned}
h(z+z^{-1})=\widehat{f}(s)&=\mathrm{H}f(0)+\sum_{n \geq 1} \mathrm{H} f(n) q^{|n| / 2} (z^n+z^{-n})\\
&=\sum_{n \geq 0}\mathrm{H}f(n)q^{|n|/2}Y_{n}(z+z^{-1}).
\end{aligned}
\end{equation}
This implies $a_{n,1}(h)=\mathrm{H}f(n)q^{|n|/2}$ for all $n\geq 0$. By Proposition \ref{xrqhuxiangzhuanhuan} and the formula \eqref{inversehoro},  we have for $n\geq 0$,
\begin{equation}
\begin{aligned}
f(n)&=\mathrm{H} f(n)-(q-1) \sum_{j \geq 1}(\mathrm{H} f)(n+2 j)\\
    &=q^{-n/2}\left(a_{n,1}(h)-\sum_{j\geq 1}q^{-j}a_{n+2j,1}(h)\right)\\
    &=q^{-n/2}\sum_{j\geq 0}q^{-j}\left(a_{n+2j,1}(h)-a_{n+2j+2,1}(h)\right)\\
    &=q^{-n/2}\sum_{j\geq 0}q^{-j}a_{n+2j,\infty}(h)=q^{-n/2}a_{n,q}(h).
\end{aligned}
\end{equation}
This completes the proof.
\end{proof}

In fact, we have the following isometry of $L^2$ spaces. 

\begin{theorem}
    For $q\in \mathbb{N}^+$, the map $$\Phi: l^2_S(\mathbb{T}_q, o) \rightarrow L^2(\mathbb{R}, \mu_q),\,\,\,\,\, f \rightarrow f(0)+(1+q^{-1})^{-1}\sum_{r=1}^{\infty}f(r)q^{r/2}X_{r,q}$$
    is an isometry, where $f(r)=f(v)$ for any $r\in \mathbb{N}$ and $v \in V_{\mathbb{T}_q}, d(o,v)=r$.
\end{theorem}

\begin{proof}
    This is due to the orthogonal relation  \eqref{orthogonalrelation}.
\end{proof}

\section{Applications}

In this section, we apply discrete trace formulas of the functional calculus for the adjacency matrix $h(A)$ (see Theorems \ref{introexpansion}, \ref{intropretrace}, and \ref{introdiscretetracecircuit}) to study various problems on regular graphs. We provide alternative derivations for several classical results, including the counting of walks, the Ihara-Bass formula, and explicit solutions to the heat and Schrödinger equations on graphs. Our objective is to demonstrate that these seemingly disparate results can be unified in one framework.

\subsection{Decompositions of holomorphic functions}

To put our trace and functional calculus formulas into practice, we show the following results. The first lemma shows how to compute the coefficients $a_{r,q}(h)$ of a holomorphic function $h$ from its Taylor series.

\begin{lemma}\label{taylortocheby}
Let $h$ and $\{a_{r,q}(h)\}$ be as in Theorem \ref{master}. Assume that $h$ is holomorphic on $B(0,R)$, $R>2$ with the Taylor series  
 $$h(\omega)=\sum_{n=0}^{\infty}b_n(h)\omega^n.$$
There holds
\begin{equation}
     a_{r,1}(h)=\sum_{k=0}^{\infty}\binom{2k+r}{k}b_{2k+r}(h).
\end{equation}
\end{lemma}

Again we delay the proof to Section \ref{section:proof}. We now decompose several frequently used complex functions along Chebyshev-type polynomials. The first is a direct consequence of Lemma \ref{xrqhuxiangzhuanhuan} and Lemma \ref{taylortocheby}.

\begin{lemma}[Decomposition of the power]\label{decompositionofpower}
    For $n\in \mathbb{N}^+$,
    there holds
    \begin{equation}
        x^n=\sum_{0\leq k\leq n/2}\binom{n}{k}Y_{n-2k}(x),
    \end{equation}
    \begin{equation}
        x^n=\sum_{0\leq k\leq n/2}\left(\binom{n}{k}-\binom{n}{k-1}\right) X_{n-2k}(x),
    \end{equation}
    and for $q\in \mathbb{N}^+$,
    \begin{equation}
        x^n=\sum_{0\leq k\leq n/2} \left(\sum_{0\leq m\leq k}
        q^{-(k-m)}\left(\binom{n}{m}-\binom{n}{m-1}\right)\right)X_{n-2k,q}(x).
    \end{equation}
\end{lemma}

Next we study the logarithm function.
\begin{lemma}[Decomposition of the logarithm]\label{logarithm}
For $|t|$ small enough, there holds
\begin{equation}
    \mathrm{log}(1-xt+t^2)=\sum_{r=1}^{\infty}\frac{Y_r(x)}{r}t^{r}.
\end{equation}
\end{lemma}
\begin{proof}
Notice that
    \begin{equation}
\begin{aligned}
    \frac{d}{dt}\mathrm{log}(1-xt+t^2)=\frac{2t-x}{1-xt+t^2}=\frac{1}{t}\left(1-\frac{1-t^2}{1-xt+t^2}\right)=\sum_{r=1}^{\infty}Y_r(x)t^{r-1}.
\end{aligned}
\end{equation}
Since $\mathrm{log}(1-xt+t^2)|_{t=0}=0$, there holds
\begin{equation}
    \mathrm{log}(1-xt+t^2)=\sum_{r=1}^{\infty}\frac{Y_r(x)}{r}t^{r}.
\end{equation}
\end{proof}
For the exponential function, we have the following decompositions.
\begin{lemma}[Decomposition of the exponential]\label{exponential}
For any $z,x\in \mathbb{C}$, there holds
\begin{equation}\label{exponentialChebyshev1}
e^{zx}=\sum_{n=0}^{\infty}I_{n}(2z)Y_n(x),
\end{equation}
\begin{equation}\label{exponentialChebyshev2}
e^{zx}=\frac{1}{z}\sum_{n=0}^{\infty}(n+1)I_{n+1}(2z)X_n(x),
\end{equation}
and for $q\in \mathbb{N}^+$,
\begin{equation}\label{exponentialChebyshevq}
e^{zx}=\frac{1}{z}\sum_{n=0}^{\infty}\left(\sum_{k=0}^{\infty}q^{-k}(n+2k+1)I_{n+2k+1}(2z)\right) X_{n,q}(x).
\end{equation}
Here $I_n$ is the modified Bessel function of the first kind of order $n$.
\end{lemma}
\begin{proof}
    Let $h_z(x)=e^{zx}$. It is an entire function, and $$h_z(x)=\sum_{n=0}^{\infty}\frac{z^n}{n!}x^n.$$
    Thus we have 
    $$
    \begin{aligned}
        a_{n,1}(h_z)=\sum_{k=0}^{\infty} \binom{n+2k}{k}\frac{z^{n+2k}}{(n+2k)!}=\sum_{k=0}^{\infty} \frac{1}{k!(k+n)!}z^{2k+n}=I_{n}(2z),
    \end{aligned}
    $$
where 
\begin{equation}\label{Besselmodified1st}
    I_{n}(z):=\sum_{k=0}^{\infty} \frac{1}{k!(k+n)!}\left(\frac{z}{2}\right)^{2k+n} 
\end{equation}
is the modified Bessel function of the first kind. 
Thus we prove (\ref{exponentialChebyshev1}).
Taking derivative with respect to $x$ of both sides of (\ref{exponentialChebyshev1}), we derive 
$$ze^{zx}=\sum_{n=0}^{\infty}(n+1)I_{n+1}(2z)X_n(x).$$
Here we use the fact from (\ref{eq:Yrderivative}) that 
$$Y_{n+1}'(z)=(n+1)X_n(z), \,n\in \mathbb{N}.$$
Thus we prove (\ref{exponentialChebyshev2}). The identity (\ref{exponentialChebyshevq}) then follows from (\ref{cheby1tocheb2}).
\end{proof}
\begin{remark}

We obtain several interesting identities of Bessel functions. Firstly, by using (\ref{cheby1tocheb2}) directly, there holds
$$e^{zx}=\sum_{n=0}^{\infty}(I_{n}(2z)-I_{n+2}(2z))X_n(x).$$
Comparing with (\ref{exponentialChebyshev1}), we derive the following well-known result that for $n\in \mathbb{N}$,
\begin{equation}\label{Besselproperty}
    I_{n}(z)-I_{n+2}(2)=\frac{2(n+1)I_{n+1}(z)}{z}.
\end{equation}
Secondly, by comparing (\ref{exponentialChebyshev1}) and the $q=1$ case of (\ref{exponentialChebyshevq}), we obtain that for $n\in \mathbb{N}$, there holds
$$zI_n(z)=\sum_{k=0}^{\infty}(n+2k+1)I_{n+2k+1}(2z).$$
This identity seems to be unnoticed before.
\end{remark}
As a direct consequence of the above decomposition, we derive the following result.
\begin{corollary}
    For any $z,x\in \mathbb{C}$, there holds 
\begin{equation}\label{wavechebyshev1}
e^{izx}=\sum_{n=0}^{\infty}i^{n}J_{n}(2z)Y_n(x),
\end{equation}

\begin{equation}\label{wavechebyshev2}
e^{izx}=\frac{1}{z}\sum_{n=0}^{\infty}(n+1)i^{n+1}J_{n+1}(2z)X_n(x),
\end{equation}
and
\begin{equation}\label{wavechebyshevq}
e^{izx}=\frac{1}{z}\sum_{n=0}^{\infty}\left(\sum_{k=0}^{\infty}(-i)^{n+2k+1}q^{-k}(n+2k+1)J_{n+2k+1}(2z)\right) X_{n,q}(x).
\end{equation}
Here $J_n$ is the Bessel function of the first kind of order $n$.
\end{corollary}

\begin{proof}
    By definition, we have $$I_n(iz)=(-i)^{n}J_n(z).$$ Then, the corollary follows directly from Lemma \ref{exponential}.
\end{proof}

\subsection{Number of walks on regular graphs}
In this subsection, we use functional calculus and trace formulas to count the number of walks on regular graphs. The following proposition is similar to Kempton \cite[Lemma 3.1]{kempton18}.

\begin{proposition}\label{numberofwalkonrg}
    Let $G=(V,E)$ be a $(q+1)$-regular graph. Denote the number of walks and non-backtracking walks from $a$ to $b$ of length $n$ as $W_n(a,b;G)$ and $f_n(a,b;G)$, respectively. Then 
    \begin{equation*}
        W_n(a,b;G)=\sum_{0\leq k\leq n/2} \left(\sum_{0\leq m\leq k}
        \left(\binom{n}{m}-\binom{n}{m-1}\right)q^{m}\right)f_{n-2k}(a,b;G).
    \end{equation*}
\end{proposition}

\begin{proof}
    Let $A$ be the adjacency matrix (operator) of $G$. Notice that 
    \[W_n(a,b;G)=\langle \mathbf{1}_{a}, A^n \mathbf{1}_{b} \rangle\] 
    and 
    \[f_n(a,b;G)=\langle \mathbf{1}_{a}, A_n \mathbf{1}_{b} \rangle.\]
    The result then follows from Lemma \ref{decompositionofpower}
    and Theorem \ref{kernel}. 
\end{proof}

The above proposition yields an explicit formula for counting walks on regular trees between any given vertices. This method (and formula) seem to be new. The counting formula of closed walks on regular trees has been obtained in McKay \cite[Lemma 2.1]{McK81}. (We remark that the third line of \cite[Lemma 2.1]{McK81} contains a typo: the summation should be taken from $k=0$ to $s$ instead of only from $k=1$ to $s$.)

\begin{corollary}\label{numberofwalktree}
    Let $\mathbb{T}_q=(V_{\mathbb{T}_q},E_{\mathbb{T}_q})$ be a $(q+1)$-regular tree with $a,b\in V_{\mathbb{T}_q}$. Let $d$ be the distance between $a,b$. Then 
    \begin{equation}
        W_{d+2k}(a,b;G)=\sum_{0\leq m\leq k}
        \left(\binom{d+2k}{m}-\binom{d+2k}{m-1}\right)q^{m}.
    \end{equation}
\end{corollary}
\begin{proof}
    We only need to notice that $f_{n}(a,b; G)=1$ if $n=d$ and $0$ otherwise.
\end{proof}

\subsection{Heat and Schr\"odinger equations on regular graphs}

Recall that for a $(q+1)$-regular $G=(V,E)$ with adjacency matrix (operator) $A$, the Laplacian is defined as $$\Delta=(q+1)I-A.$$ The heat and Schr\"odinger equations on $G$ are defined respectively to be
$$\left\{
\begin{aligned}
    &\frac{\partial}{\partial t}u(t)=-\Delta u(t),\\
    &u(0)=v,\,\,\,v \in \ell^2(V);
\end{aligned} \right.$$ 
and 
$$\left\{
\begin{aligned}
    &i \frac{\partial}{\partial t}u(t)=\Delta u(t),\\
    &u(0)=v,\,\,\,v \in \ell^2(V).
\end{aligned} \right.$$ 
The first one depicts the continuous time random walk and the second one serves as the approximate quantum model of certain molecules. These equations are solved respectively by 
$$u(t)=e^{-t\Delta}v$$
and 
$$u(t)=e^{-it\Delta}v.$$
 We now use our functional calculus formula to express the operators $e^{-t\Delta}$ and $e^{it\Delta}$ via the non-backtracking matrices (operators). Similar work can be found in Chung and Yau \cite{CY99}, Chinta, Jorgenson, and Karlsson \cite{CJK15}.

\begin{proposition}\label{fdsolutionofheat}
Let $G=(V,E)$ be a $(q+1)-$regular graph with Laplacian $\Delta$. There holds 
\begin{equation}\label{heatkernelbound}
    e^{-t\Delta} =h_{0,q} (t)\mathrm{I} +\sum_{r=1}^{\infty} h_{r,q}(t)A_r,
\end{equation}
and
\begin{equation}
    e^{-it\Delta} =w_{0,q} (t)\mathrm{I} +\sum_{r=1}^{\infty} w_{r,q}(t)A_r,
\end{equation}
where
\begin{equation}\label{eq:heat}
    h_{r,q}(t)=\frac{e^{-(q+1)t}}{t}\sum_{k=0}^{\infty}
   q^{-(r+2k+1)/2}(r+1+2k)I_{r+1+2k}(2q^{1/2}t),
\end{equation}
and
\begin{equation}\label{wave}
    w_{r,q}(t)=-\frac{e^{-i(q+1)t}}{t}\sum_{k=0}^{\infty}
   q^{-(r+2k+1)/2}{-i}^{r+2k}(r+1+2k)J_{r+1+2k}(2q^{1/2}t).
\end{equation}
\end{proposition}
\begin{remark}\label{lattice1}
    In particular, for $q=1$, we have 
    $$h_{r,1}(t)=e^{-2t}I_r(2t),$$
    and
    $$w_{r,1}(t)=(-i)^ne^{2it}J_r(2t).$$
\end{remark}
\begin{proof}
    From (\ref{exponentialChebyshevq}), we know that
    \begin{equation*}
    \begin{aligned}
e^{-t\Delta}&=e^{-(q+1)t}\cdot \frac{1}{q^{1/2}t}\sum_{n=0}^{\infty}\left(\sum_{k=0}^{\infty}q^{-k}(n+2k+1)I_{n+2k+1}(2q^{1/2}t)\right) X_{n,q}(q^{-1/2}A)\\
&=\frac{e^{-(q+1)t}}{t}\sum_{n=0}^{\infty}\left(\sum_{k=0}^{\infty}q^{-k-1/2}(n+2k+1)I_{n+2k+1}(2q^{1/2}t)\right) q^{-r/2}A_r\\
&=\frac{e^{-(q+1)t}}{t}\sum_{n=0}^{\infty}\left(\sum_{k=0}^{\infty}q^{-(r+2k+1)/2}(n+2k+1)I_{n+2k+1}(2q^{1/2}t)\right) A_r.
\end{aligned}
\end{equation*}
Thus there holds
$$
\begin{aligned}
e^{-it\Delta}&=\frac{e^{i(q+1)t}}{it}\sum_{n=0}^{\infty}\left(\sum_{k=0}^{\infty}q^{-(r+2k+1)/2}(n+2k+1)I_{n+2k+1}(2q^{1/2}it)\right) A_r\\
&=\frac{e^{i(q+1)t}}{it}\sum_{n=0}^{\infty}\left(\sum_{k=0}^{\infty}q^{-(r+2k+1)/2}(-i)^{n+2k+1}(n+2k+1)J_{n+2k+1}(2q^{1/2}t)\right) A_r.\\
\end{aligned}
$$
\end{proof}

As a consequence, the fundamental solutions on regular trees can be explicitly expressed as in the following corollary.

\begin{corollary}\label{heattree}
Let $\mathbb{T}_q=(V_{\mathbb{T}_q}, E_{\mathbb{T}_q})$ be the $(q+1)$-regular tree with Laplacian $\Delta_q$. The entries of $e^{-t\Delta}$ and $e^{it\Delta}$ are then given by 
\begin{equation}
    \left( e^{-t\Delta} \right)_{u,v}=h_{d(u,v),q}(t), \,\,\,u,v \in V_{\mathbb{T}_q},
\end{equation}
and 
\begin{equation}
    \left( e^{it\Delta} \right)_{u,v}=w_{d(u,v),q}(t),\,\,\,u,v \in V_{\mathbb{T}_q},
\end{equation}
where $d(u,v)$ is the distance between $u$ and $v$.
\end{corollary}

We now derive the following estimate for the heat equations on regular graphs.
\begin{proposition}\label{coefficientofeq}
Let $G=(V,E)$ be a $(q+1)-$regular graph, $q>1$, with Laplacian $\Delta$. Let $h_{r,q}(t)$ be defined as in (\ref{heatkernelbound}). Then there holds 
 $$(r+1)\frac{I_{r+1}(2q^{1/2}t)}{te^{(q+1)t}} \leq h_{r,q}(t)\leq (1-q^{-1})^{-2}(r+1)\frac{I_{r+1}(2q^{1/2}t)}{te^{(q+1)t}}.$$
\end{proposition}
\begin{proof}
    When $t\geq 0,$ we have $I_n(t)>0$ for $n\geq 0$, by the Taylor expansion (\ref{Besselmodified1st}) of modified Bessel function. This implies the inequality on the left. For the inequality on the right, notice that by the positiveness of the Bessel functions, the property (\ref{Besselproperty}) implies $I_n(t)>I_{n+2}(t)$. Thus there holds
    $$\sum_{k=0}^{\infty}
   q^{-k}(r+1+2k)I_{r+1+2k}(t)\leq \left(\sum_{k=0}^{\infty}
   q^{-k}(r+1+2k)\right)I_{r+1}(t).$$
   A direct computation yields that
   $$\sum_{k=0}^{\infty}
   q^{-k}(r+1+2k)\leq (1-q^{-1})^{-2}(r+1).$$
   Thus the proof is completed.
\end{proof}

\begin{remark}
    A different but equivalent form of (\ref{eq:heat}) as below 
    \begin{equation*}
        h_{r,q}(t)=q^{-r/2}e^{-(q+1)t}I_r(2q^{1/2}t)-(q-1)\sum_{k=1}^{\infty}q^{-(r+2k)/2}e^{-(q+1)t}I_{r+2k}(2q^{1/2}t)
    \end{equation*}
    has been established by Chinta, Jorgenson, and Karlsson \cite{CJK15}. 
\end{remark}

\subsection{Heat equations on lattices and applications}

Let $\{\mathbb{Z}^D, \{\pm1\}^{D} \}$ be the Cayley graph of the free Abelian group $\mathbb{Z}^D$. We refer to it as the $D$-dimesional (grid) lattice and denote it as $L^D$ for short. Notice that the 1-dimensional lattice  $L^1$ is the 2-regular tree, and the $d$-dimesional lattice is the $d$-times Cartesian product of $L^1$. To study the heat equation on lattices, we need the following result related to the graph Cartesian product. Similar work can be found in Chung and Yau \cite{CY97}.

\begin{lemma}\label{graphprod}
    Let $G=G_1 \times G_2$ be the Cartesian product of two graphs $G_1$ and $G_2$. Denote the adjacency matrix and Laplacian of $G$ as $A$ and $\Delta$. For $z\in \mathbb{C}$, there holds
\begin{equation}
    e^{zA}=e^{zA_1}\otimes e^{zA_2}, \quad e^{z\Delta}=e^{z\Delta_1}\otimes e^{z\Delta_2},
\end{equation}
where for $k=1,2$, $A_k$ and $\Delta_k$ are the adjacency matrix and Laplacian of $G_k.$
\end{lemma}
\begin{proof}
    Notice that 
    $$A=A_1 \otimes \mathrm{I}+\mathrm{I} \otimes A_2,$$
    and 
    $$\Delta=\Delta_1 \otimes \mathrm{I}+\mathrm{I} \otimes\Delta_2.$$
Since $$(A_1 \otimes \mathrm{I}) \cdot (\mathrm{I} \otimes A_2)=(\mathrm{I} \otimes A_2)\cdot(A_1 \otimes \mathrm{I}),$$
and 
$$(\Delta_1 \otimes \mathrm{I}) \cdot (\mathrm{I} \otimes \Delta_2)=(\mathrm{I} \otimes \Delta_2)\cdot(\Delta_1 \otimes \mathrm{I}),$$
our result holds.
\end{proof}

Denote the adjacency operator and Laplacian of $L^D$ as $A_{L^D}$ and $\Delta_{L^D}$, respectively. By Lemma \ref{graphprod}, we can calculate $e^{-t\Delta_{L^D}}$ and $e^{-tA_{L^D}}$ directly from Remark \ref{lattice1} and Corollary \ref{heattree} as follows.

\begin{proposition}\label{heatkerneloflattice}
    Let $\mathbf{m}=(m_1, m_2,\cdots,m_D)$ and $\mathbf{n}=(n_1, n_2,\cdots, n_D)$ be two vertices on $L^D$. Then the following equality holds
    \begin{equation}
        \left(e^{-t\Delta_{L^D}}\right)_{\mathbf{m},\mathbf{n}}=e^{-2Dt}\prod_{k=1}^{d}I_{|m_k-n_k|}(2t).
    \end{equation}
\end{proposition}

\begin{proof}
We compute directly as follows:
$$
\begin{aligned}
    \left(e^{-t\Delta_{L^D}}\right)_{\mathbf{m},\mathbf{n}}&=\langle \mathbf{1}_{\mathbf{m}}, e^{-t\Delta_{L^D}}\mathbf{1}_{\mathbf{n}} \rangle\\
    &=\left\langle \otimes_{k=1}^D\mathbf{1}_{m_k}, \otimes^D e^{-t\Delta_{L^1}} \otimes_{k=1}^D\mathbf{1}_{n_k}\right\rangle\\
    &=\prod_{k=1}^D \left\langle \mathbf{1}_{m_k}, e^{-t\Delta_{L^1}} \mathbf{1}_{n_k} \right\rangle=\prod_{k=1}^D \left(e^{-t\Delta_{L^1}}\right)_{m_k,n_k}.
\end{aligned}
$$
Due to Remark \ref{lattice1} and Corollary \ref{heattree}, for any $m,n\in \mathbb{Z}$,
$$\left(e^{-t\Delta_{L^1}}\right)_{m,n}=e^{-2t}I_{|m-n|}(2t).$$
The proof is then completed.
\end{proof}
In the same way, the following equality holds.
\begin{proposition}\label{walkgenerating}
    Let $\mathbf{m}=(m_1, m_2,\cdots,m_D)$ and $\mathbf{n}=(n_1, n_2,\cdots, n_D)$ be two vertices on $L^D$. Then the following holds
    \begin{equation}
        \left(e^{tA_{L^D}}\right)_{\mathbf{m}, \mathbf{n}}=\prod_{k=1}^{D}I_{|m_k-n_k|}(2t).
    \end{equation}
\end{proposition}

The above equality could be used to count the number of walks on lattices.

\begin{corollary}\label{numberofwalklattice}
    Let $\mathbf{m}=(m_1, m_2,\cdots,m_D)$ and $\mathbf{n}=(n_1, n_2,\cdots, n_D)$ be two vertices on $L^D$. Let $$d=\sum_{l=1}^D |m_l-n_l|$$ be the distance of $\mathbf{m}, \mathbf{n}$ in $L^D$. Then there holds
    $$W_{d+2k}(\mathbf{m},\mathbf{n};L^D)=\sum_{\substack{k_1+k_2+\cdots+k_D=k,\\ k_1, k_2, \cdots, k_D \in \mathbb{N}}}(d+2k)!\prod_{l=1}^{D}\frac{1}{k_l! (k_l+|m_l-n_l|)!}.$$ 
\end{corollary}

\begin{proof}
    Notice that 
    $$
    \begin{aligned}
        \left(e^{tA_{L^D}}\right)_{\mathbf{m}, \mathbf{n}}=\sum_{l=0}^{\infty}\frac{1}{l!}\left(A_{L^D}^l\right)_{\mathbf{m}, \mathbf{n}}t^l&=\sum_{k=0}^{\infty}\frac{1}{(d+2k)!}\left(A_{L^D}^{d+2k}\right)_{\mathbf{m}, \mathbf{n}}t^{d+2k}\\
        &=\sum_{k=0}^{\infty}\frac{1}{(d+2k)!}W_{d+2k}(\mathbf{m},\mathbf{n};L^D) t^{d+2k}.
    \end{aligned}
    $$
The result then follows directly from Proposition \ref{walkgenerating} and the Taylor series of  $I_n$, i.e., the equation (\ref{Besselmodified1st}).
\end{proof}
\subsection{Ihara-Bass theorem, Fourier transform and heat trace}

In this subsection, we apply our trace formula to show the Ihara-Bass Theorem.
Let $G$ be a finite $(q+1)$-regular graph. Recall $\mathcal{P}(G)$ is the collection of equivalent classes of prime circuits in $G$. For $t$ small enough, the \textbf{Ihara $\zeta$ function} is defined to be 
    \begin{equation}
        \zeta_G(t):=\prod_{\gamma\in \mathcal{P}(G)}\frac{1}{1-t^{\ell_\gamma}}.
    \end{equation}

The Ihara-Bass formula, see \cite{BH92}, connects the Ihara $\zeta$ function to the adjacency matrix of a regular graph.

\begin{theorem}[Ihara-Bass]\label{iharabasseq} Let $G=(V,E)$ be a $(q+1)$-regular graph with adjacency matrix $A$. There holds
    \begin{equation}
        \frac{1}{\zeta_G(t)}=(1-t^2)^{\frac{1}{2}(q-1)|V|}\mathrm{det}(I-tA+t^2I).
    \end{equation}
\end{theorem}
\begin{proof}
Recall from Lemma \ref{logarithm} that 
\begin{equation*}
    \mathrm{log}(1-xt+t^2)=\sum_{r=1}^{\infty}\frac{Y_r(x)}{r}t^{r}.
\end{equation*}
Thus, we derive 
\begin{equation*}
\begin{aligned}
     \int_{\mathbb{R}}\mathrm{log}(1-xt+t^2)d\mu_q(x)&=\sum_{r=1}^{\infty}\frac{t^{r}}{r}\int_{\mathbb{R}}Y_r(x)d\mu_q=\sum_{r=1}^{\infty}\frac{t^{2r}}{2r}\int_{\mathbb{R}}Y_{2r}(x)d\mu_q\\
     &=\sum_{r=1}^{\infty}\frac{t^{2r}}{2r}\int_{\mathbb{R}}\left(X_{2r}(x)-X_{2r-2}(x)\right)d\mu_q\\
     &=\sum_{r=1}^{\infty}\frac{t^{2r}}{r}(q^{-r}-q^{-r+1})=\frac{1}{2}(1-q)\log\left(1-\frac{t^2}{q}\right).
\end{aligned}
\end{equation*}
By Theorem \ref{discretetracecircuit}, there holds
    $$\sum_{k=1}^{|V|} \log(1-tq^{-1/2}\lambda_k(A)+t^2)=
    \frac{1}{2}(1-q)\log(1-\frac{t^2}{q})|V_G|+\sum_{r=1}^{\infty}\frac{c_r(G)}{r}(q^{-1/2}t)^{r},$$for small enough $t$.
    This implies
    $$\sum_{k=1}^{|V|} \log(1-t\lambda_k(A)+qt^2)=
    \frac{1}{2}(1-q)\log(1-t^2)|V|+\sum_{r=1}^{\infty}\frac{c_r(G)}{r}t^{r}.$$
Recalling that $c_r=\sum_{\gamma \in \mathcal{P}(G),\ell_\gamma|r}\ell_\gamma$, we conclude
$$
\begin{aligned}
    \sum_{r=1}^{\infty}\frac{c_r(G)}{r}t^{r}&=\sum_{r=1}^{\infty}\sum_{\gamma,\ell_\gamma|r}\frac{\ell_\gamma}{r}t^{\ell_\gamma\frac{r}{\ell_\gamma}}=\sum_{\gamma \in \mathcal{P}(G)} \sum_{k=1}^{\infty} \frac{1}{k}(t^{\ell_\gamma})^k=\sum_{\gamma \in \mathcal{P}(G)} \log(1-t^{\ell_\gamma}).
\end{aligned}
$$
\end{proof}

The next application is the \emph{Laplace-Fourier transform of spectral measures of regular graphs}. For any probability measure $\mu$ on $\mathbb{R}$, we define the Fourier-Laplace transform of $\mu$ is given by 
\[\hat{\mu}(p):=\int_{\mathbb{R}} e^{ipx}d\mu(x), \,\,\text{for}\,\,p\in \mathbb{C}.\] 

\begin{theorem}[Fourier-Laplace transform of spectral measure]\label{thm:fourierlaplace}
   Suppose $G$ is a $(q+1)\text{-}$regular graph with adjacency matrix (operator) $A$. Let $\mu_G$ be the normalized spectral measure of $G$. Then 
   
    \begin{equation}\label{Fourier-Laplace}
     \hat{\mu}_G(p)=\hat{\mu}_q(p)+\sum_{r=1}^{\infty}q^{-r/2}(i)^rJ_r(2p)\frac{c_r(G)}{|V_G|},
    \end{equation}
 The right hand side converges absolutely and uniformly on any compact subset of $\mathbb{C}$ with respect to $p$.
\end{theorem}

\begin{proof}
    Fix $p\in \mathbb{C}$. The equality ($\ref{Fourier-Laplace}$) follows directly from the trace formula Theorem \ref{discretetracecircuit} and (\ref{wavechebyshev1}). To see the uniform convergence on any compact subset with respect to $p$, notice that for $t>0$, 
    \begin{equation*}
    \begin{aligned}
      I_{n}(t)&=\sum_{k=0}^{\infty} \frac{1}{k!(k+n)!}\left(\frac{t}{2}\right)^{2k+n}=\left(\frac{t}{2}\right)^n\sum_{k=0}^{\infty} \frac{\binom{n+2k}{n}}{(n+2k)!}\left(\frac{t}{2}\right)^{2k}  \\
      &\leq \left(\frac{t}{2}\right)^n\sum_{k=0}^{\infty} \frac{1}{(n+2k)!}t^{2k}\leq  \left(\frac{t}{2}\right)^n\frac{1}{n!}\sum_{k=0}^{\infty} \frac{1}{k!k!}t^{2k}=\left(\frac{t}{2}\right)^n\frac{1}{n!} I_0(2t).
    \end{aligned}
\end{equation*}
For all $p\in \mathbb{C}$ with $|p|\leq R$, we have
$$
\begin{aligned}
\left|\sum_{r=1}^{\infty}q^{-r/2}(i)^rJ_r(2p)\frac{c_r(G)}{|V_G|}\right|
&\leq\sum_{r=1}^{\infty}q^{-r/2}|J_r(2p)|\frac{(1+q^{-1})q^{r}}{|V_G|}\\
&\leq\sum_{r=1}^{\infty}I_r(|2p|)\frac{(1+q^{-1})q^{r/2}}{|V_G|}\\
&\leq \frac{(q+1)I_0(4R)}{|V_G|} \sum_{r=1}^{\infty}\frac{1}{r!}\left(\frac{q^{1/2}R}{2}\right)^r<+\infty.
\end{aligned}
$$
Our proof is then completed.
\end{proof}

By the definition of the Bessel function, $I_n$ has a zero of order $n$ at $0$. Thus we have the following corollary.

\begin{corollary}
    Let $G$ be a finite $(q+1)$-regular graph. Then the girth of $G$ is the order of zero of $\hat{\mu}_G-\hat{\mu}_q$ at $p=0$.
\end{corollary}

Another direct corollary of Theorem \ref{thm:fourierlaplace} is the following heat trace formula for regular graphs. This is also proved by Mn\"ev \cite{Mnev07}, Horton, Newland and Terras \cite{HNT06}, Chinta, Jorgenson, and Karlsson \cite{CJK15}.

\begin{corollary}[Heat trace formula]\label{heattraceformula}
Let $G=(V,E)$ be a finite $(q+1)$-regular graph with Laplacian $\Delta$. Let $0=\lambda_1(\Delta) \leq \lambda_2(\Delta) \leq \cdots \lambda_{|V|}(\Delta)$ be the eigenvalues of $\Delta$. The following heat trace formula holds:
\begin{equation*}
\sum_{k=1}^{|V|}e^{-\lambda_kt}=\int_{\mathbb{R}}e^{-t\left(q^{1/2}-1\right)^2 x} d\mu_q(x)+\sum_{r=1}^\infty q^{-r/2}\frac{c_r(G)}{|V|}e^{-(q+1)t}I_n(2q^{1/2}t),
    \end{equation*}
where $I_n$ is the modified Bessel function of the first kind of order $n$.
\end{corollary}

\section{Proof of the technical Lemmas}\label{section:proof}
In this section, we provide the proofs for Theorem \ref{master}, Lemma \ref{xrqhuxiangzhuanhuan}
and Lemma \ref{taylortocheby}.
\begin{proof}[Proof of Theorem \ref{master}]
    Fix $z \in \Omega(\rho)$. Choose $0<\varepsilon<\rho-1$ such that $z \in \Omega(\rho-\varepsilon)$. By the Cauchy's integral formula, we have 
    $$h(z)=\frac{1}{2\pi i}\oint_{\partial \Omega(\rho-\varepsilon)}\frac{h(\omega)}{\omega-z}d\omega.$$
    By a change of variable $\omega=\xi+\xi^{-1}$, we have
    \begin{equation}
        \begin{aligned}
    h(z)&=\frac{1}{2\pi i}\oint_{\partial \Omega(\rho-\varepsilon)}\frac{h(\omega)}{\omega-z}d\omega\\
    &=\frac{1}{2\pi i}\oint_{\partial B(0, \frac{1}{\rho-\varepsilon})}\frac{h(\xi+\xi^{-1})}{\xi+\xi^{-1}-z}(\frac{1}{\xi^2}-1)d\xi\\
    &=\frac{1}{2\pi i}\oint_{\partial B(0, \frac{1}{\rho-\varepsilon})} \frac{1-\xi^2}{\xi(1-q^{-1}\xi^2)} h(\xi+\xi^{-1}) \frac{1-q^{-1}\xi^2}{1-z\xi+\xi^2}d\xi.
\end{aligned}
    \end{equation}
    Notice that by the choice of $\varepsilon$, 
    $$\frac{1-q^{-1}\xi^2}{1-z\xi+\xi^2}$$ has no poles on $\overline{ B(0, \frac{1}{\rho-\varepsilon})}$ as a function of $\xi$. By (\ref{Xrqgeneratingfunction}), 
    $$\frac{1-q^{-1}\xi^2}{1-z\xi+\xi^2}=\sum_{r=0}^{\infty}X_{r,q}(z)\xi^r,$$
    and the right hand side converges uniformly on $\overline{ B(0, \frac{1}{\rho-\varepsilon})}$. Thus 
    $$
    \begin{aligned}
    h(z)&=\frac{1}{2\pi i}\oint_{\partial B(0, \frac{1}{\rho-\varepsilon})} \frac{1-\xi^2}{\xi(1-q^{-1}\xi^2)} h(\xi+\xi^{-1}) \sum_{r=0}^{\infty}X_{r,q}(z) \xi^r d\xi\\
    &=\sum_{r=0}^{\infty} \left( \frac{1}{2\pi i}\oint_{\partial B(0, \frac{1}{\rho-\varepsilon})} h(\xi+\xi^{-1})\xi^{r-1}\frac{1-\xi^2}{1-q^{-1}\xi^2}d\xi \right) X_{r,q}(z).
\end{aligned}$$
Here we can interchange the order of integral and summation due to the uniform convergence mentioned above.
For $q>1$, the function 
$$h(\xi+\xi^{-1})\xi^{r-1}\frac{1-\xi^2}{1-q^{-1}\xi^2}$$ is holomorphic on $\{z:\rho^{-1}<|z|<\mathrm{min}(\rho, q^{1/2}) \}$. Thus 
$$
\begin{aligned}
    &\frac{1}{2\pi i}\oint_{\partial B(0, \frac{1}{\rho-\varepsilon})} h(\xi+\xi^{-1})\xi^{r-1}\frac{1-\xi^2}{1-q^{-1}\xi^2}d\xi\\
    &=\frac{1}{2\pi i}\oint_{\partial B(0, 1)} h(\xi+\xi^{-1})\xi^{r-1}\frac{1-\xi^2}{1-q^{-1}\xi^2}d\xi.
\end{aligned}
$$
For $q=1$, the coefficient of $X_{r,1}$ equals
$$\frac{1}{2\pi i}\oint_{\partial B(0, \frac{1}{\rho-\varepsilon})} h(\xi+\xi^{-1})\xi^{r-1}d\xi.$$
Similarly, we have 
$$
\begin{aligned}
    \frac{1}{2\pi i}\oint_{\partial B(0, \frac{1}{\rho-\varepsilon})} h(\xi+\xi^{-1})\xi^{r-1}d\xi=\frac{1}{2\pi i}\oint_{\partial B(0, 1)} h(\xi+\xi^{-1})\xi^{r-1}d\xi.
\end{aligned}
$$

Next, we prove the absolute and uniform convergence. Fix any compact subset $K\subset \Omega(\rho)$. There exits $0<\varepsilon_K<\rho-1$ such that $K\subset \Omega(\rho-\varepsilon_K)$. For any $z\in \Omega(\rho-\varepsilon_K)$, there exists $t\in \mathbb{C}$, such that $1\leq |t|\leq \rho-\varepsilon_K$ amd $z=t+t^{-1}$.
Since $$
    X_r(z)=\frac{t^{r+1}-t^{-r-1}}{t-t^{-1}}=\sum_{i=0}^{r}t^{r-2i},
$$
we have 
$$X_{r,q}(z)=X_r(z)-q^{-1}X_{r-2}(z)=t^{r}+t^{-r}+(1-q^{-1})\sum_{i=0}^{r}t^{r-2i}.$$
Hence
$$|X_{r,q}(z)|\leq |t|^{r}+|t|^{-r}+(1-q^{-1})\sum_{i=0}^{r}|t|^{r-2i}\leq (r+1)|t|^{r}\leq (r+1) (\rho-\varepsilon_K)^r.$$
For $q>1$, as shown above,
$$a_{r,q}(h)=\frac{1}{2\pi i}\oint_{\partial B(0, \tau)} h(\xi+\xi^{-1})\xi^{r-1}\frac{1-\xi^2}{1-q^{-1}\xi^2}d\xi,$$
for all $ \rho^{-1}<\tau\leq 1$. Thus 
$$|a_{r,q}|\leq M(\tau) \tau^r$$
where $$M(\tau)=\sup_{|\xi|=\tau}\ h(\xi+\xi^{-1})\frac{1-\xi^2}{1-q^{-1}\xi^2}.$$
Choose $\rho^{-1} <\tau<(\rho-\varepsilon_K)^{-1}$, there holds for any $z\in K$
$$\left|\sum_{r=0}^{\infty} a_{r,q}(h) X_{r,q}(z)\right|\leq \sum_{r=0}^{\infty} M(\tau) (r+1) \left(\frac{\tau}{\rho-\varepsilon_K}\right) ^r<+\infty.$$
The same proof can be applied to the case $q=1$. Thus the convergence is absolute and uniform on the compact subset $K$. 
\end{proof}

\begin{proof}[Proof of Proposition \ref{xrqhuxiangzhuanhuan}]
    By Theorem \ref{master}, for all $q \in \mathbb{N}\cup \{\infty\}$, 
    $h$ can be decomposed as 
    $$h(z)=\sum_{r=0}^{\infty}a_{r,q}(h)X_{r,q}(z),$$
    and the right hand side converges uniformly on $[-2,2]$. In this way, by the orthogonal relation (\ref{orthogonalrelation}), there holds
    \begin{equation*}
    \begin{aligned}
        a_{m,\infty}(h)&=\langle h, X_m\rangle_{\mu_{\infty}}\\
        &=\langle \sum_{n=0}^{\infty}a_{n,1}(h)Y_n, X_m\rangle_{\mu_{\infty}}\\
        &=\sum_{n=0}^{\infty}a_{n,1}(h)\langle X_n-X_{n-2}, X_m\rangle_{\mu_{\infty}}\\
        &=a_{m,1}(h)-a_{m+2,1}(h).
    \end{aligned}
\end{equation*}
and
\begin{equation*}
\begin{aligned}
    a_{m,q}(h)&=||X_{m,q}||_{\mu_q}^{-2}\langle h, X_{m,q} \rangle _{\mu_q}\\
    &=||X_{m,q}||_{\mu_q}^{-2}\sum_{n=0}^{\infty}a_{n,\infty}(h)\langle X_n, X_{m,q} \rangle _{\mu_q}\\
    &=||X_{m,q}||_{\mu_q}^{-2}\sum_{n=0}^{\infty}a_{n,\infty}(h) \left\langle
    \sum_{k=0}^{\lfloor n/2\rfloor}q^{-k}X_{n-2k,q}, X_{m,q} \right\rangle _{\mu_q},\\
    &=\sum_{k=0}^{\infty} q^{-k} a_{m+2k,\infty}(h).
\end{aligned}
\end{equation*}
    
\end{proof}

\begin{proof}[Proof of Lemma \ref{taylortocheby}]
    It suffices to calculate
    $\oint_{\partial B(0,1)} h(\xi+\xi^{-1})\xi^{r-1}d\xi $. Since $h$ is holomorphic on $B(0,R)$ and $R>2$, there exits $C, \delta>0$ such that $|b_n|\leq C (2+\delta)^{-n}$ for $n\in \mathbb{N}$. For $|\xi|=1$,
    \begin{equation*}
    \begin{aligned}
        h(\xi+\xi^{-1})&=\sum_{n=0}^{\infty}b_n(\xi+{\xi}^{-1})^n\\
        &=\sum_{n=0}^{\infty}b_n\sum_{k=0}^n\binom{n}{k}\xi^{n-2k}\\
        &=\sum_{m=-\infty}^{\infty}\left(\sum_{k=0}^{\infty}b_{2k+|m|}\binom{2k+|m|}{k}\right)\xi^m.
    \end{aligned}
\end{equation*}
Here we can interchange the order of summations due to
$$
\begin{aligned}
&\sum_{m=-\infty}^{\infty}\sum_{k=0}^{\infty}\left|b_{2k+|m|}\binom{2k+|m|}{k}\xi^m\right|\\
&\leq \sum_{m=-\infty}^{\infty}\sum_{k=0}^{\infty} C (2+\delta)^{-2k-|m|}2^{2k+|m|}<\infty.
\end{aligned}
$$
In the above, we use the bound $\binom{m}{n}\leq 2^m$. Notice that this argument also implies that the right hand side of 
$$h(\xi+\xi^{-1})=\sum_{m=-\infty}^{\infty}\left(\sum_{k=0}^{\infty}b_{2k+|m|}\binom{2k+|m|}{k}\right)\xi^m$$
converges absolutely on $\partial B(0,1)$. Thus
$$
\begin{aligned}
    a_{r,1}(h)&=\oint_{\partial B(0,1)} h(\xi+\xi^{-1})\xi^{r-1}d\xi\\
    &=\sum_{m=-\infty}^{\infty}\sum_{k=0}^{\infty}b_{2k+|m|}\binom{2k+|m|}{k}\oint_{\partial B(0,1)} z^{m+r-1}dz=\sum_{k=0}^{\infty}\binom{2k+r}{k}b_{2k+r}.
\end{aligned}
$$
The proof is complete.
\end{proof}

\noindent \textbf{Acknowledgements:} YG was supported by the National Key R \& D Program of China 2022YFA1007400. WL is supported by the National Natural Science Foundation of China No. 123B2013. SL is supported by the National Key R \& D Program of China 2023YFA1010200 and the National Natural Science Foundation of China No. 12031017. The authors would like to thank Yunhui Wu and Yuxin He for numerous useful discussions, and are also grateful to the anonymous referee for several helpful suggestions that improved the article.

\bibliographystyle{plain}
\bibliography{reference}

\end{document}